\documentclass[11pt]{article}

\usepackage{mathpazo}
\usepackage{amsfonts}
\usepackage{amsmath}
\usepackage{bm}
\usepackage{graphicx}
\usepackage{latexsym}
\usepackage[margin=1.05in]{geometry}
  
\usepackage[T1]{fontenc}
\usepackage{times}
\usepackage{color,graphicx}
\usepackage{array}
\usepackage{enumerate}
\usepackage{amsmath}
\usepackage{amssymb}
\usepackage{amsthm}
\usepackage{bbm} 
\usepackage{pgfplots}
\usepackage{pgf}
\usepackage{tikz}
\usetikzlibrary{patterns}
\usetikzlibrary{arrows.meta}
\usepgfplotslibrary{patchplots} 
\usetikzlibrary{pgfplots.patchplots} 
\pgfplotsset{width=9cm,compat=1.5.1}
\usepackage{diagbox}

\usepackage{amsthm}

\newtheorem{lemma}{Lemma}
\newtheorem{theorem}{Theorem}
\newtheorem{corollary}{Corollary}
\newtheorem{remark}{Remark}

\newtheorem{example}{Example} 

\usepackage{xcolor}
\usepackage{makeidx}
\usepackage[colorlinks=true,linkcolor=blue,anchorcolor=blue,citecolor=red,urlcolor=magenta]{hyperref}
\usepackage{caption}
\usepackage{subcaption}

\newcommand{\tinyspace}{\mspace{1mu}}

\newcommand{\tr}{\operatorname{Tr}}
\newcommand{\rank}{\operatorname{rank}}

\newcommand{\norm}[1]{\left\lVert\tinyspace#1\tinyspace\right\rVert}

\newcommand{\defeq}{\stackrel{\smash{\textnormal{\tiny def}}}{=}}

\def\C{\mathbb{C}}

\def\R{\mathbb{R}} 

\def\S{\mathcal{S}}

\def\e{\mathbf{e}}

\def\v{\mathbf{v}}
\def\w{\mathbf{w}}
\def\x{\mathbf{x}}
\def\y{\mathbf{y}}

\def\0{\mathbf{0}}

\begin{document}

\title{Bounding Real Tensor Optimizations via the Numerical Range}  

\author{
	Nathaniel Johnston\footnote{Department of Mathematics \& Computer Science, Mount Allison University, Sackville, NB, Canada E4L 1E4}\textsuperscript{$\ \ $}\footnote{Department of Mathematics \& Statistics, University of Guelph, Guelph, ON, Canada N1G 2W1} \ \ \ and \ \ 
	Logan Pipes$^{*}$
}


\maketitle

\begin{abstract}
    We show how the numerical range of a matrix can be used to bound the optimal value of certain optimization problems over real tensor product vectors. Our bound is stronger than the trivial bounds based on eigenvalues, and can be computed significantly faster than bounds provided by semidefinite programming relaxations. We discuss numerous applications to other hard linear algebra problems, such as showing that a real subspace of matrices contains no rank-one matrix, and showing that a linear map acting on matrices is positive.
\end{abstract}
  

\section{Introduction}\label{sec:intro}

The numerical range \cite[Chapter~1]{HJ91} of a square complex matrix $A \in M_n(\C)$ is the set
\[
    W(A) \defeq \big\{\mathbf{v}^*A\mathbf{v} : \mathbf{v} \in \C^n, \|\mathbf{v}\| = 1 \big\}.
\]
While it was first introduced and studied over a century ago \cite{Toe18,Hau19}, the numerical range has come into particular prominence over the past two decades or so, thanks to numerous applications that it has found in scientific fields like numerical analysis \cite{Ben21} and quantum information theory \cite{KPLRS09,CHCZ20,CRZ22}.

In this work, we present a new application of the numerical range: we show how it can be used to give a non-trivial bound on the optimal value of a linear function being optimized over the set of real tensor product vectors (i.e., elementary tensors) of unit length. Computing the exact value of such an optimization problem is NP-hard \cite{Gur03},\footnote{The proof of \cite[Proposition~6.5]{Gur03} focuses on optimizing over \emph{complex} product vectors, but the same proof works for real product vectors.} and numerous well-studied problems can be phrased in this way. We thus obtain, as immediate corollaries of our results, ways of using the numerical range to help solve the following (also NP-hard) problems:

\begin{itemize}
    \item Determine whether or not a given subspace of real matrices contains a rank-1 matrix \cite{CP86,NSU17}.

    \item Determine whether or not a given linear map acting on $M_n(\R)$ is positive \cite{Pau03,CDPR22}.
    
    \item Determine whether or not a biquadratic form acting on $\R^m \times \R^n$ is positive semidefinite \cite{Cho75}.
\end{itemize}

While our method is specific to optimizations over \emph{real} product vectors, it can be extended with some limitations to optimizations over complex product vectors as well. This gives some additional applications of our method, particularly from quantum information theory, such as the problem of determining whether or not a given Hermitian matrix is an entanglement witness \cite{Ter00}. In fact, in \cite{HM13} the authors give a list of more than a dozen different problems that are of interest in quantum information theory that are equivalent to this one.

We note that connections between the numerical range and the above problems have appeared in the past, but our bound is new and distinct from known methods. For example, an easy-to-use necessary condition for positivity of a linear map acting on complex matrices in terms of the numerical range appeared in \cite{HLP15}, but our bound instead provides an easy-to-use \emph{sufficient} condition. Similarly, numerous other methods are known for bounding optimization problems like the ones we investigate \cite{DPS04,HNW17}. However, most of these known methods rely on tools like semidefinite programming which, despite being solvable in polynomial time \cite{Lov03}, are quite computationally expensive. These semidefinite programs are typically difficult to apply to matrices with more than a few hundred rows and columns, whereas our bound can be applied to matrices with millions of rows and columns.

\subsection{Organization of Paper}\label{sec:organization}

We start in Section~\ref{sec:prelim} by introducing our notation, the optimization problem that is of interest to us, and some elementary properties of its optimal value. In Section~\ref{sec:main_result}, we present our main results (Theorems~\ref{thm:mu_minmax_from_num_range} and \ref{thm:W_diag_equals_p}), which provide a non-trivial bound on the optimal value of the optimization problem in terms of the numerical range of a certain matrix. In Section~\ref{sec:comparison_with_known} we discuss how our method compares with known semidefinite programming methods of bounding the optimal value of our tensor optimization problem (implementation details of our method are provided in Appendix~\ref{sec:implementation}, and code is provided at \cite{Pip22}), and we show in particular that our method works much faster and is applicable to matrices that are much larger.

In Section~\ref{sec:applications} we present some applications of our bound: in Section~\ref{sec:matrix_subspaces} we show that our result implies a simple-to-check sufficient condition for a subspace of matrices to not contain any rank-one matrices, and in Section~\ref{sec:pos_lin_map} we derive a simple-to-check sufficient condition for a linear map acting on matrices to be positive (or equivalently, for a biquadratic form to be positive semidefinite). We then generalize our results in two ways: in Section~\ref{sec:complex} we generalize our results to optimizations over complex product vectors (with some necessary limitations), and in Section~\ref{sec:multipartite} we generalize our results to the setting where more than two spaces are tensored together.

\section{Optimization over Product Vectors}\label{sec:prelim}

We use $M_{m,n}$ to denote the set of $m \times n$ matrices with real entries, $M_n$ for the special case when $m = n$, and $M_n^{\textup{S}}$ for the subset of them that are symmetric (i.e., the matrices $A \in M_n$ satisfying $A^T = A$). In the few instances where we consider sets of complex matrices, we denote them by $M_{m,n}(\C)$ or $M_n(\C)$, as appropriate. We use ``$\otimes$'' to denote the usual tensor (i.e., Kronecker) product, and $\mathbf{e_j}$ denotes the $j$-th standard basis vector of $\R^n$ (i.e., the vector with $1$ in its $j$-th entry and $0$ elsewhere).

The main problem that we consider in this paper is: given $B \in M_m \otimes M_n$, how can we compute or bound the quantities
\begin{align}\begin{split}\label{eq:main_minmax}
    \mu_{\textup{min}}(B) & \defeq \min_{\mathbf{v} \in \R^m,\mathbf{w} \in \R^n}\big\{ (\mathbf{v} \otimes \mathbf{w})^TB(\mathbf{v} \otimes \mathbf{w}) : \|\mathbf{v}\| = \|\mathbf{w}\| = 1 \big\} \quad \text{and} \\
    \mu_{\textup{max}}(B) & \defeq \max_{\v \in \R^m,\w \in \R^n}\big\{ (\v \otimes \mathbf{w})^TB(\v \otimes \w) : \|\v\| = \|\w\| = 1 \big\}?
\end{split}\end{align}

When $B$ is symmetric, these quantities can be thought of as analogous to the maximal and minimal eigenvalues of $B$, but with the restriction that the optimization takes place only over unit-length \emph{product} vectors, rather than all unit-length vectors. In particular, if $B$ is symmetric and we let $\lambda_{\textup{max}}(B)$ and $\lambda_{\textup{min}}(B)$ denote the maximal and minimal eigenvalues of $B$, respectively, then it follows immediately from the Rayleigh quotient that
\begin{align}\label{eq:eigenvalue_bounds}
    \mu_{\textup{min}}(B) \geq \lambda_{\textup{min}}(B) \quad \text{and} \quad \mu_{\textup{max}}(B) \leq \lambda_{\textup{max}}(B).
\end{align}

Similarly, if we define the \emph{partial transpose} of a matrix $A = \sum_j X_j \otimes Y_j \in M_m \otimes M_n$ by
\begin{align}\label{eq:part_trans}
    A^\Gamma \defeq \sum_j X_j \otimes Y_j^T,
\end{align}
and note that the partial transpose is a well-defined linear map on $M_m \otimes M_n$,\footnote{The partial transpose has received significant attention when acting on complex matrices thanks to its uses in quantum information theory \cite{Per96}, but it has recently been used in the real case too \cite{CDPR22}.} then we see that
\begin{align}\begin{split}\label{eq:vw_part_trans}
    (\v \otimes \w)^T B(\v \otimes \w) & = \tr\big((\v\v^T \otimes \w\w^T)B\big) \\
    & = \tr\big((\v\v^T \otimes (\w\w^T)^T) B\big) \\
    & = \tr\big((\v\v^T \otimes \w\w^T)B^\Gamma\big) = (\v \otimes \w)^TB^\Gamma(\v \otimes \w)
\end{split}\end{align}
for all $\v \in \R^m$ and $\w \in \R^n$. It follows that if $B$ is symmetric then
\begin{align}\label{eq:pt_eigenvalue_bounds}
    \mu_{\textup{min}}(B) \geq \lambda_{\textup{min}}(B^\Gamma) \quad \text{and} \quad \mu_{\textup{max}}(B) \leq \lambda_{\textup{max}}(B^\Gamma).
\end{align}

We call the eigenvalue bounds of Inequalities~\eqref{eq:eigenvalue_bounds} and~\eqref{eq:pt_eigenvalue_bounds} the \emph{trivial bounds} on $\mu_{\textup{max}}(B)$ and $\mu_{\textup{min}}(B)$, and an argument similar to the one that we used to derive them shows that if $p \in \R$ then
\begin{align}\label{eq:conv_bounds}
    \mu_{\textup{min}}(B) \geq \lambda_{\textup{min}}\big(pB + (1-p)B^\Gamma\big) \quad \text{and} \quad \mu_{\textup{max}}(B) \leq \lambda_{\textup{max}}\big(pB + (1-p)B^\Gamma\big)
\end{align}
whenever $B$ is symmetric.

Inequalities~\eqref{eq:conv_bounds} can only be used when $B$ is symmetric, since otherwise $pB + (1-p)B^\Gamma$ might not even have real eigenvalues. However, this is not really a restriction: since $\v^TY\v = 0$ whenever $Y$ is skew-symmetric (i.e., $Y^T = -Y$), it follows that if $B$ has Cartesian decomposition $B = X + Y$ (where $X = (B+B^T)/2$ is symmetric and $Y = (B-B^T)/2$ is skew-symmetric) then
\begin{align}\label{eq:sym_part_equal}
    \mu_{\textup{min}}(B) = \mu_{\textup{min}}(X) \quad \text{and} \quad \mu_{\textup{max}}(B) = \mu_{\textup{max}}(X),
\end{align}
so Inequalities~\eqref{eq:conv_bounds} can be applied to $X = (B+B^T)/2$ instead. In fact, the following lemma shows that we can restrict our attention even further to matrices that also equal their own \emph{partial} transpose:

\begin{lemma}\label{lem:bisymmetric_part_real}
    Let $B \in M_m \otimes M_n$ be symmetric. There exist unique symmetric matrices $X, Y \in M_m \otimes M_n$ such that $X^\Gamma = X$, $Y^\Gamma = -Y$, and $B = X + Y$, and they are
    \begin{align}\label{eq:partial_sym_parts}
        X = \frac{1}{2}(B + B^\Gamma) \quad \text{and} \quad Y = \frac{1}{2}(B - B^\Gamma).
    \end{align}
    Furthermore, $(\v \otimes \w)^T Y(\v \otimes \w) = 0$ for all $\v \in \R^m$ and $\w \in \R^n$, and
    \[
        \mu_{\textup{min}}(B) = \mu_{\textup{min}}(X) \quad \text{and} \quad \mu_{\textup{max}}(B) = \mu_{\textup{max}}(X).
    \]
\end{lemma}

\begin{proof}
    It is clear that the matrices $X$ and $Y$ from Equation~\eqref{eq:partial_sym_parts} are symmetric and satisfy $X^\Gamma = X$, $Y^\Gamma = -Y$, and $B = X+Y$. To see that they are unique, just notice that $B = X+Y$ implies $B^\Gamma = X^\Gamma + Y^\Gamma = X - Y$, so $B + B^\Gamma = (X+Y) + (X-Y) = 2X$, so $X = (B + B^\Gamma)/2$, as claimed. Uniqueness of $Y$ can similarly be derived by computing $B - B^\Gamma$.
    
    To see that $(\v \otimes \w)^T Y(\v \otimes \w) = 0$ for all $\v \in \R^m$ and $\w \in \R^n$, we just note that the same argument used in Equation~\eqref{eq:vw_part_trans} tells us that
    \[
        (\v \otimes \w)^T (-Y)(\v \otimes \w) = (\v \otimes \w)^T Y^\Gamma(\v \otimes \w) = (\v \otimes \w)^T Y(\v \otimes \w).
    \]
    Since $B = X + Y$, this immediately implies $\mu_{\textup{max}}(B) = \mu_{\textup{max}}(X)$ and $\mu_{\textup{min}}(B) = \mu_{\textup{min}}(X)$.
\end{proof}

The above lemma can be thought of as a ``partial'' version of the usual Cartesian decomposition, and when combined with the usual Cartesian decomposition it tells us that for any $B \in M_m \otimes M_n$ we have
\[
    \mu_{\textup{min}}(B) = \frac{1}{4}\mu_{\textup{min}}\big(B + B^T + B^\Gamma + (B^T)^\Gamma\big) \quad \text{and} \quad \mu_{\textup{max}}(B) = \frac{1}{4}\mu_{\textup{max}}\big(B + B^T + B^\Gamma + (B^T)^\Gamma\big).
\]

\section{Main Results}\label{sec:main_result}

We now present our main pair of results, which bound the optimal values of the optimization problems~\eqref{eq:main_minmax} in an easily-computable way in terms of the numerical range. Given $B \in M_m \otimes M_n$, consider the following set (recall that $W(B + iB^\Gamma)$ is the numerical range of the complex matrix $B+iB^\Gamma$):
\begin{align}\label{eq:def_W_diag}
    W^{1+i}(B) \defeq \big\{c \in \R : c(1+i) \in W(B + iB^\Gamma)\big\}.
\end{align}

This set is always non-empty, since it contains all of the diagonal entries of $B$, for example. To see this, notice that if $\mathbf{e_j}$ is the $j$-th standard basis vector of $\R^n$ ($1 \leq j \leq n$), and we choose $\mathbf{x} = \mathbf{e_j} \otimes \mathbf{e_k}$ for some $1 \leq j \leq m$ and $1 \leq k \leq n$, then
\[
    \x^*(B+iB^\Gamma)\x = \x^*B\x + i\x^*B^\Gamma\x = b_{(j,j),(k,k)}(1 + i),
\]
where $b_{(j,j),(k,k)}$ denotes the $(k,k)$-entry of the $(j,j)$-block of $B$. The set $W^{1+i}(B)$ is furthermore convex and compact since $W(A)$ is convex and compact \cite[Section~1.3]{HJ91} for all complex matrices $A$. It follows that $W^{1+i}(B)$ is a closed and bounded interval of real numbers. We give names to its endpoints:
\begin{align}\label{eq:W_diag_bounds}
    W^{1+i}_{\textup{min}}(B) \defeq \min\big\{c \in W^{1+i}(B)\big\} \quad \text{and} \quad W^{1+i}_{\textup{max}}(B) \defeq \max\big\{c \in W^{1+i}(B)\big\}.
\end{align}

Our first main result describes how $W^{1+i}_{\textup{min}}(B)$ and $W^{1+i}_{\textup{max}}(B)$ relate to $\mu_{\textup{min}}(B)$ and $\mu_{\textup{max}}(B)$.

\begin{theorem}\label{thm:mu_minmax_from_num_range}
    Let $B \in M_m \otimes M_n$. Then
    \begin{align*}
        \mu_{\textup{min}}(B) \geq W^{1+i}_{\textup{min}}(B) \quad \text{and} \quad \mu_{\textup{max}}(B) \leq W^{1+i}_{\textup{max}}(B).
    \end{align*}
\end{theorem}

\begin{proof}
    We only prove the inequality on the right, as the one on the left follows via an almost identical argument.
    
    Thanks to compactness of the set of unit product vectors, there exists $\x = \v \otimes \w \in \R^m \otimes \R^n$ so that $\|\x\| = 1$ and $\x^T B\x = \mu_{\textup{max}}(B)$. Routine algebra then shows that
    \[
        \x^T B^{\Gamma}\x = \tr\big(B^{\Gamma}\x\x^T\big) = \tr\big(B(\x\x^T)^{\Gamma}\big) = \tr\big(B\x\x^T\big) = \x^T B\x = \mu_{\textup{max}}(B)
    \]
    too, where the central equality used the fact that $\mathbf{x}$ is a real product vector, so
    \[
        (\mathbf{x}\mathbf{x}^T)^{\Gamma} = \mathbf{v}\mathbf{v}^T \otimes (\mathbf{w}\mathbf{w}^T)^T = \mathbf{v}\mathbf{v}^T \otimes \mathbf{w}\mathbf{w}^T = \mathbf{x}\mathbf{x}^T.
    \]
    It follows that
    \[
        \mathbf{x}^T (B + iB^\Gamma)\mathbf{x} = \mathbf{x}^T B\mathbf{x} + i\mathbf{x}^T B^\Gamma\mathbf{x} = (1+i)\mu_{\textup{max}}(B) \in W(B + iB^\Gamma),
    \]
   so $\mu_{\textup{max}}(B) \in W^{1+i}(B)$, which shows that $\mu_{\textup{max}}(B) \leq W^{1+i}_{\textup{max}}(B)$, as desired.
\end{proof}


Our second main result establishes how the bound of Theorem~\ref{thm:mu_minmax_from_num_range} compares with the eigenvalue bounds of Inequalities~\eqref{eq:conv_bounds}. We do not prove this theorem here, however, as it arises as a special case of a more general result that we will prove in Section~\ref{sec:multipartite} (see Theorem~\ref{thm:multi_affine_minmax} in particular).
\begin{theorem}\label{thm:W_diag_equals_p}
    Let $B \in M_m \otimes M_n$ be symmetric. Then
    \begin{align*}
        W^{1+i}_{\textup{min}}(B) = \max_{p \in \R}\big\{\lambda_{\textup{min}}(pB + (1-p)B^\Gamma)\big\} \quad \text{and} \quad W^{1+i}_{\textup{max}}(B) = \min_{p \in \R}\big\{\lambda_{\textup{max}}(pB + (1-p)B^\Gamma)\big\}.
    \end{align*}
\end{theorem}

Another way of phrasing Theorem~\ref{thm:mu_minmax_from_num_range} is as saying that $\mu_{\textup{min}}(B),\mu_{\textup{max}}(B) \in W^{1+i}(B)$. Similarly, Theorem~\ref{thm:W_diag_equals_p} says exactly that, if $B$ is symmetric, then $W^{1+i}(B)$ is the following intersection of subintervals of $\R$:
\[
    W^{1+i}(B) = \bigcap_{p \in \R} \big[ \lambda_{\textup{min}}(pB + (1-p)B^\Gamma), \lambda_{\textup{max}}(pB + (1-p)B^\Gamma)\big].
\]
These facts are illustrated schematically in Figure~\ref{fig:main_theorem_schem}.

\begin{figure}[htb]
\begin{centering}
	\begin{tikzpicture}[yscale=1,xscale=1.33333333]
		\foreach \x in {-5,-4,...,5} {
			\draw[ultra thin, gray!50!white] (\x, -3.4) -- (\x, 3.4);
		}

		\foreach \y in {-3,-2,...,3} {
			\draw[ultra thin, gray!50!white] (-5.375, \y) -- (5.375, \y);
		}

		\draw[->] (-5.5, 0) -- (5.6, 0) node[anchor=west] {$\mathrm{Re}$}; 
		\draw[->] (0, -3.5) -- (0, 3.6) node[anchor=south] {$\mathrm{Im}$}; 
		
		\draw[black] (-3.4,-3.4) -- (3.4,3.4) node[anchor=south west]{$\{c(1+i) : c \in \R\}$}; 

		\draw[fill=gray!50, fill opacity=0.8]
(2.328652, -0.642872) -- (2.326706, -0.735382) -- (2.320740, -0.830075) -- (2.310563, -0.926824) -- (2.296002, -1.025380) -- (2.276921, -1.125366) -- (2.253244, -1.226289) -- (2.224964, -1.327559) -- (2.192159, -1.428515) -- (2.154989, -1.528461) -- (2.113692, -1.626704) -- (2.068572, -1.722597) -- (2.019976, -1.815562) -- (1.968276, -1.905120) -- (1.913846, -1.990901) -- (1.857041, -2.072646) -- (1.798183, -2.150201) -- (1.737550, -2.223506) -- (1.675372, -2.292573) -- (1.611825, -2.357475) -- (1.547039, -2.418322) -- (1.481100, -2.475246) -- (1.414055, -2.528392) -- (1.345923, -2.577898) -- (1.276702, -2.623893) -- (1.206375, -2.666487) -- (1.134925, -2.705770) -- (1.062335, -2.741806) -- (0.988602, -2.774635) -- (0.913741, -2.804275) -- (0.837792, -2.830723) -- (0.760825, -2.853961) -- (0.682940, -2.873958) -- (0.604274, -2.890679) -- (0.524991, -2.904089) -- (0.445285, -2.914158) -- (0.365368, -2.920870) -- (0.285458, -2.924220) -- (0.205771, -2.924221) -- (0.126501, -2.920900) -- (0.047804, -2.914293) -- (-0.030224, -2.904437) -- (-0.107568, -2.891355) -- (-0.184316, -2.875041) -- (-0.260696, -2.855428) -- (-0.337115, -2.832351) -- (-0.414225, -2.805490) -- (-0.493008, -2.774283) -- (-0.574929, -2.737786) -- (-0.662144, -2.694455) -- (-0.757813, -2.641803) -- (-0.866449, -2.575923) -- (-0.994008, -2.491043) -- (-1.146786, -2.379870) -- (-1.327364, -2.236562) -- (-1.527740, -2.063489) -- (-1.727511, -1.875923) -- (-1.905425, -1.694397) -- (-2.052352, -1.531408) -- (-2.170700, -1.388517) -- (-2.266956, -1.261835) -- (-2.346982, -1.146789) -- (-2.414839, -1.039938) -- (-2.473054, -0.939171) -- (-2.523150, -0.843402) -- (-2.566091, -0.752205) -- (-2.602565, -0.665495) -- (-2.633153, -0.583309) -- (-2.658404, -0.505662) -- (-2.678857, -0.432476) -- (-2.695038, -0.363566) -- (-2.707439, -0.298640) -- (-2.716507, -0.237326) -- (-2.722628, -0.179192) -- (-2.726124, -0.123764) -- (-2.727246, -0.070536) -- (-2.726171, -0.018972) -- (-2.722997, 0.031505) -- (-2.717739, 0.081535) -- (-2.710306, 0.131846) -- (-2.700481, 0.183311) -- (-2.687869, 0.237026) -- (-2.671821, 0.294433) -- (-2.651298, 0.357506) -- (-2.624644, 0.429054) -- (-2.589213, 0.513179) -- (-2.540743, 0.615961) -- (-2.472458, 0.746273) -- (-2.374137, 0.916162) -- (-2.232606, 1.138686) -- (-2.037419, 1.419057) -- (-1.794433, 1.738959) -- (-1.534953, 2.052770) -- (-1.298947, 2.315272) -- (-1.108344, 2.510314) -- (-0.963396, 2.646740) -- (-0.854704, 2.740773) -- (-0.772193, 2.806310) -- (-0.708034, 2.853016) -- (-0.656750, 2.887150) -- (-0.614624, 2.912704) -- (-0.579137, 2.932243) -- (-0.548554, 2.947446) -- (-0.521661, 2.959435) -- (-0.497585, 2.968979) -- (-0.475687, 2.976613) -- (-0.455487, 2.982718) -- (-0.436616, 2.987569) -- (-0.418782, 2.991364) -- (-0.401752, 2.994247) -- (-0.385333, 2.996323) -- (-0.369358, 2.997666) -- (-0.353684, 2.998324) -- (-0.338179, 2.998324) -- (-0.322720, 2.997676) -- (-0.307188, 2.996371) -- (-0.291463, 2.994384) -- (-0.275419, 2.991668) -- (-0.258919, 2.988159) -- (-0.241811, 2.983763) -- (-0.223918, 2.978356) -- (-0.205028, 2.971773) -- (-0.184880, 2.963788) -- (-0.163141, 2.954100) -- (-0.139373, 2.942290) -- (-0.112974, 2.927760) -- (-0.083077, 2.909631) -- (-0.048362, 2.886533) -- (-0.006666, 2.856187) -- (0.045923, 2.814416) -- (0.117406, 2.752544) -- (0.226911, 2.649324) -- (0.427309, 2.443634) -- (0.842704, 1.979916) -- (1.419672, 1.282825) -- (1.763303, 0.832602) -- (1.917100, 0.612380) -- (2.002967, 0.477470) -- (2.062988, 0.373682) -- (2.110587, 0.282721) -- (2.150681, 0.197566) -- (2.185348, 0.115128) -- (2.215579, 0.033859) -- (2.241896, -0.047121) -- (2.264591, -0.128400) -- (2.283827, -0.210416) -- (2.299681, -0.293539) -- (2.312163, -0.378091) -- (2.321224, -0.464353) -- (2.326768, -0.552558) -- (2.328652, -0.642872) -- cycle;

		\draw[thick] (-1.801879, -1.801879) -- (1.358992, 1.358992); 
		\draw[dashed] (-1.801879, -5.1) -- (-1.801879, -1.801879); 
		\node[circle, fill, inner sep=1.5pt] at (-1.801879, -1.801879) {};
		\draw[dashed] (1.358992, -5.1) -- (1.358992, 1.358992); 
		\node[circle, fill, inner sep=1.5pt] at (1.358992, 1.358992) {};

        \node[anchor=west] at (2.328652, -0.642872) {$\lambda_\textup{max}(B)$}; 
        \draw[dashed] (2.328652, -3.4) -- (2.328652, -0.642872);
        \node[circle, fill, inner sep=1.5pt] at (2.328652, -0.642872) {};
		\node[anchor=east, shift={(-0.05,0.3)}] at (-0.345918, 2.998405) {$\lambda_\textup{max}(B^\Gamma)$}; 
		\draw[dashed,black!40!white] (2.998405, -1) -- (2.998405, -0.3);
		\draw[dashed] (2.998405, -3.4) -- (2.998405, -1);
		\draw[dashed] (2.998405, -0.3) -- (2.998405, 2.998405) -- (-0.345918, 2.998405);
		\node[circle, fill, inner sep=1.5pt] at (-0.345918, 2.998405) {};
		
        \node[anchor=east,shift={(0,-0.3)}] at (-2.727246, -0.070536) {$\lambda_\textup{min}(B)$}; 
        \draw[dashed] (-2.727246, -3.4) -- (-2.727246, -0.070536);
        \node[circle, fill, inner sep=1.5pt] at (-2.727246, -0.070536) {};
		
        \node[anchor=south,shift={(-0.3,0.1)}] at (0.22, -2.924221) {$\lambda_\textup{min}(B^\Gamma)$}; 
        \draw[dashed] (-2.924221, -3.4) -- (-2.924221, -2.924221) -- (0.22, -2.924221);
        \node[circle, fill, inner sep=1.5pt] at (0.22, -2.924221) {};

        \node[anchor=north] at (-1.1, -3.6) {$\mu_\textup{min}(B)$};
        \draw[dashed] (-1.2, -1.2) -- (-1.2, -3.65);
        \node[circle, fill, inner sep=1.5pt] at (-1.2, -1.2) {};
        
        \node[anchor=north] at (0.7, -3.6) {$\mu_\textup{max}(B)$};
        \draw[dashed] (0.6, 0.6) -- (0.6, -3.65);
        \node[circle, fill, inner sep=1.5pt] at (0.6, 0.6) {};
        
		\node at (-1.2, 0.7) {$W(B+iB^\Gamma)$};
		\draw[thick, decorate, decoration={brace,amplitude=10pt}] (-1.801879, -3.5) -- (-5.375,-3.5);
		\node[anchor=north] at (-3.843827,-3.85) {$\lambda_\textup{min}(pB+(1-p)B^\Gamma)$};
		\draw[thick, decorate, decoration={brace,amplitude=10pt}] (1.358992, -4.5) -- (-1.801879, -4.5);
		\node[anchor=north] at (-0.15,-4.85) {$W^{1+i}(B)$};
		\draw[thick, decorate, decoration={brace,amplitude=10pt}]  (5.375,-3.5) -- (1.358992, -3.5);
		\node[anchor=north] at (3.819075,-3.85) {$\lambda_\textup{max}(pB+(1-p)B^\Gamma)$};
		\node[anchor=north, shift={(0,-0.1)}] at (1.358992, -5.1) {$W^{1+i}_\textup{max}(B)$};
		\node[anchor=north, shift={(0,-0.1)}] at (-1.801879, -5.1) {$W^{1+i}_\textup{min}(B)$};
	\end{tikzpicture}
\caption{An illustration of the bounds described by Theorems~\ref{thm:mu_minmax_from_num_range} and~\ref{thm:W_diag_equals_p} when $B \in M_m \otimes M_n$ is symmetric.}\label{fig:main_theorem_schem}
\end{centering}
\end{figure}
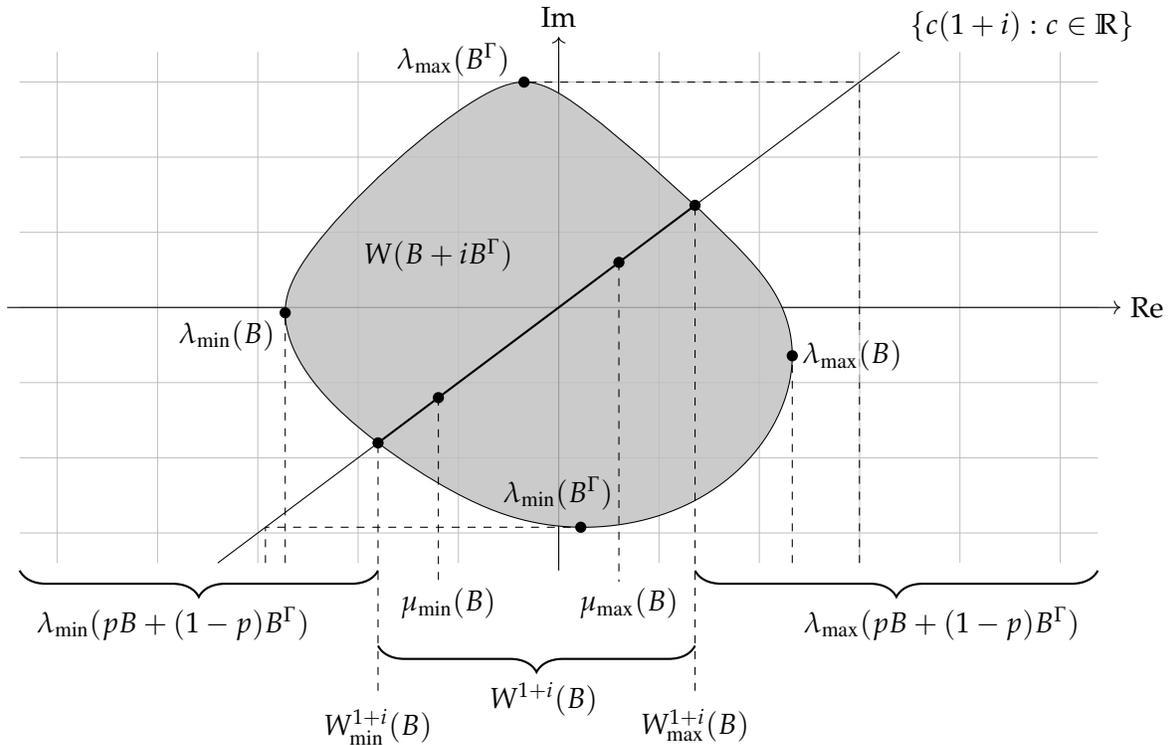

\begin{remark}\label{rem:why1i}
    There is nothing particularly special about the ``$1+i$'' in $W^{1+i}(B)$. We could have instead picked any two non-zero $y,z \in \R$ and replaced the defining Equation~\eqref{eq:def_W_diag} with
    \[
        W^{y+iz}(B) \defeq \big\{c \in \R : c(y+iz) \in W(yB + izB^\Gamma)\big\}.
    \]
    As long as $y,z \neq 0$, we have $W^{1+i}(B) = W^{y+iz}(B)$ (after all, $c \in W^{y+iz}(B)$ if and only if there exists a unit vector $\x$ for which $\x^*B\x = \x^*B^\Gamma\x$, which is also equivalent to $c \in W^{1+i}(B)$), so the bounds provided by Theorems~\ref{thm:mu_minmax_from_num_range} and~\ref{thm:W_diag_equals_p} would be unchanged if we replaced $W^{1+i}(B)$ by $W^{y+iz}(B)$. We picked $x = y = 1$ for simplicity and numerical stability.
    
    (On the other hand, if $z = 0$ then $W^{y}(B) = [\lambda_{\textup{min}}(B),\lambda_{\textup{max}}(B)]$ and if $y = 0$ then $W^{iz}(B) = [\lambda_{\textup{min}}(B^\Gamma),\lambda_{\textup{max}}(B^\Gamma)]$, so in these cases we would recover just the trivial bounds on $\mu_{\textup{min}}(B)$ and $\mu_{\textup{max}}(B)$ in Theorem~\ref{thm:mu_minmax_from_num_range}.)
\end{remark}

\subsection{Comparison with Known Bounds}\label{sec:comparison_with_known}

Besides providing a useful theoretical link between tensor optimization and the numerical range, Theorems~\ref{thm:mu_minmax_from_num_range} and~\ref{thm:W_diag_equals_p} are also useful practically since they are much easier to compute than other known bounds. In this section, we briefly compare our bounds with other techniques that are typically employed to bound the optimal value of tensor optimization problems.

Typical methods of bounding $\mu_\textup{min}(B)$ and $\mu_\textup{max}(B)$ make use of semidefinite programming (SDP) relaxations. Wwe do not provide a detailed introduction to semidefinite programming here---see \cite{BV04} or \cite[Section~3.C]{JohALA} for such an introduction---but rather we just note that semidefinite programs are a specific type of convex optimization problem that can be solved in polynomial time \cite{Lov03}. For example, it is straightforward to show that the optimal value of the following semidefinite program (in the symmetric matrix variable $X \in M_m \otimes M_n$) provides an upper bound on $\mu_\textup{max}(B)$:
\begin{align}\begin{split}\label{eq:sdp_max_WA}
    \textup{maximize:} & \ \tr(BX) \\
    \textup{subject to:} & \ X^\Gamma = X \\
    & \ \tr(X) = 1 \\
    & \ X \succeq O,
\end{split}\end{align}
where the constraint $X \succeq O$ means that $X$ is (symmetric) positive semidefinite. Indeed, for any unit vectors $\v$ and $\w$ in the optimization problem~\eqref{eq:main_minmax}, we can choose $X = \v\v^T \otimes \w\w^T$ as a feasible point in the SDP~\eqref{eq:sdp_max_WA} that attains the same value in the objective function.

In fact, semidefinite programs like this one typically provide a better bound on $\mu_\textup{max}(B)$ than Theorem~\ref{thm:mu_minmax_from_num_range} does. The advantage of Theorem~\ref{thm:mu_minmax_from_num_range} is that it requires significantly fewer computational resources to implement. Despite the fact that semidefinite programs can be solved in polynomial time, in practice they are quite slow and use an extraordinary amount of memory when the matrices involved have more than a few hundred rows and columns. By comparison, the numerical range of a matrix (and thus the bound provided by Theorem~\ref{thm:mu_minmax_from_num_range}) can easily be computed for any matrix that fits into a computer's memory---tens of thousands of rows and columns, or even millions of rows and columns if the matrix is sparse (i.e., has most entries equal to $0$).

\begin{example}\label{exam:random_large_full}
    If $m = n = 19$ then the semidefinite program~\eqref{eq:sdp_max_WA} takes about $45$ minutes to solve when applied to a random symmetric matrix $B \in M_m \otimes M_n$, whereas our method finishes in about $1$~second (we implemented our method in MATLAB and our code is available at~\cite{Pip22}).
    
    In fact, the semidefinite program~\eqref{eq:sdp_max_WA} is impractical to run already when $m = n = 20$, as our desktop computer with $16$Gb of RAM runs out of memory before completing the computation in this case via standard semidefinite programming solvers in either of MATLAB~\cite{cvx} or Julia~\cite{convexjl}. By contrast, our method continues working for significantly larger matrices, and takes about $17$~minutes to apply to a random symmetric matrix $B \in M_m \otimes M_n$ when $m = n = 100$.
\end{example}

\begin{example}\label{exam:random_large_sparse}
    Since the numerical range of a matrix can be computed by just finding the maximum eigenvalue of some Hermitian matrices (see Appendix~\ref{sec:implementation} for details), and there are algorithms for numerically computing the maximum eigenvalue of large sparse matrices very quickly, Theorem~\ref{thm:mu_minmax_from_num_range} can be directly applied to large sparse matrices.
    
    For example, when $m = n = 500$ we generated a random symmetric sparse matrix $B \in M_m \otimes M_n$ (with entries generated independently according to a standard normal distribution) with approximately $1,000,000$ non-zero entries. For this matrix, our code computed both of the bounds from Theorem~\ref{thm:mu_minmax_from_num_range} in a total of $6$ minutes, showing that $\mu_{min}(B) \geq W^{1+i}_\textup{min}(B) \approx -1.9094$ and $\mu_{max}(B) \leq W^{1+i}_\textup{max}(B) \approx 2.3991$. These bounds are much better than the trivial eigenvalue bounds of $\mu_{min}(B) \geq \lambda_\textup{min}(B) \approx -3.2537$ and $\mu_{max}(B) \leq \lambda_\textup{max}(B) \approx 3.2538$ (see Figure~\ref{fig:large_C}).
    
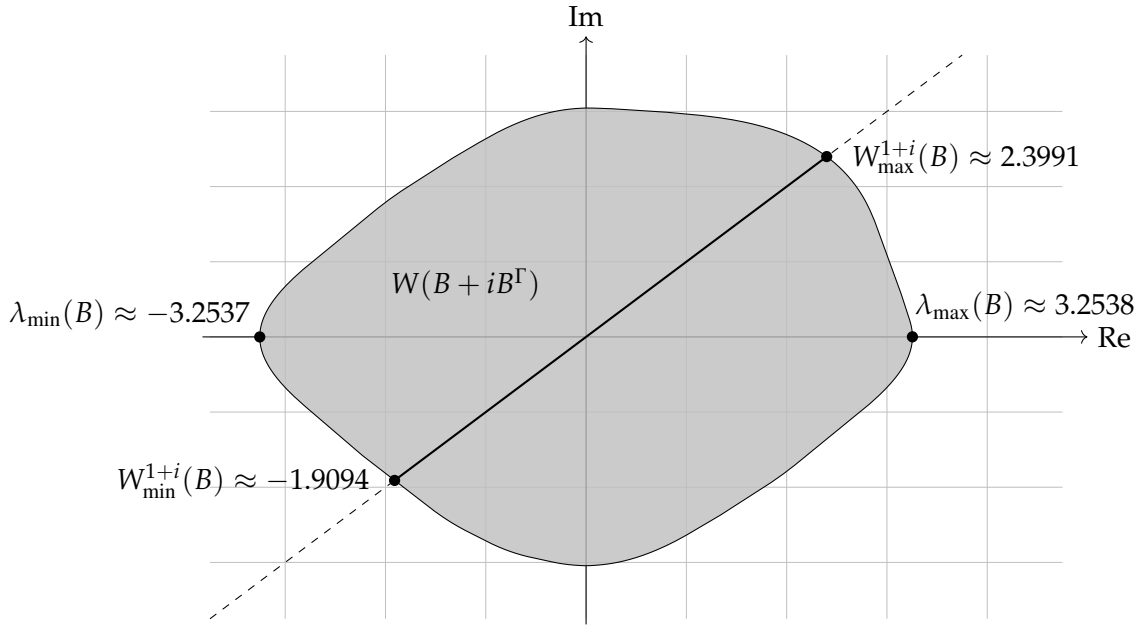
\begin{figure}[hbt]
\begin{centering}
	\begin{tikzpicture}[yscale=1,xscale=1.33333333] 
		\foreach \x in {-3,...,4} {
			\draw[ultra thin, gray!50!white] (\x, -3.75) -- (\x, 3.75);
		}

		\foreach \y in {-3,...,3} {
			\draw[ultra thin, gray!50!white] (-3.75, \y) -- (4.75, \y);
		}

		\draw [->] (-3.825, 0) -- (5, 0) node[anchor=west] {$\mathrm{Re}$}; 
		\draw [->] (0, -3.825) -- (0, 4) node[anchor=south] {$\mathrm{Im}$}; 

		\draw [dashed] (-3.75,-3.75) -- (3.75,3.75); 

		\draw [fill=gray!50, fill opacity=0.8]
(3.253810, 0.001177) -- (3.253070, -0.034171) -- (3.250840, -0.069588) -- (3.247080, -0.105344) -- (3.241708, -0.141699) -- (3.234604, -0.178917) -- (3.225599, -0.217282) -- (3.214470, -0.257113) -- (3.200923, -0.298777) -- (3.184573, -0.342705) -- (3.164919, -0.389422) -- (3.141296, -0.439581) -- (3.112812, -0.494014) -- (3.078246, -0.553824) -- (3.035868, -0.620528) -- (2.983114, -0.696335) -- (2.915860, -0.784808) -- (2.824565, -0.894902) -- (2.080696, -1.706196) -- (1.965678, -1.823671) -- (1.844541, -1.937408) -- (1.337607, -2.372392) -- (0.999657, -2.637784) -- (0.894060, -2.714657) -- (0.802373, -2.775633) -- (0.718959, -2.826196) -- (0.642207, -2.868426) -- (0.570999, -2.903802) -- (0.504454, -2.933453) -- (0.441839, -2.958263) -- (0.382535, -2.978930) -- (0.326007, -2.996009) -- (0.271788, -3.009940) -- (0.219458, -3.021071) -- (0.168626, -3.029675) -- (0.112262, -3.036675) -- (0.049082, -3.042539) -- (0.016614, -3.043900) -- (-0.015774, -3.043900) -- (-0.048621, -3.042519) -- (-0.122683, -3.036064) -- (-0.173244, -3.029682) -- (-0.224164, -3.021064) -- (-0.276583, -3.009914) -- (-0.331007, -2.995928) -- (-0.588134, -2.921296) -- (-0.643557, -2.901989) -- (-0.700005, -2.879632) -- (-0.757672, -2.853947) -- (-0.816769, -2.824599) -- (-0.877545, -2.791172) -- (-0.940324, -2.753134) -- (-1.203046, -2.575660) -- (-1.275603, -2.522929) -- (-1.349989, -2.463943) -- (-1.427489, -2.397007) -- (-1.510147, -2.319320) -- (-2.226688, -1.581371) -- (-2.805532, -0.919330) -- (-2.912227, -0.791169) -- (-2.980679, -0.701125) -- (-3.034173, -0.624255) -- (-3.077064, -0.556746) -- (-3.111993, -0.496308) -- (-3.140734, -0.441382) -- (-3.164539, -0.390839) -- (-3.184319, -0.343824) -- (-3.200752, -0.299672) -- (-3.214351, -0.257849) -- (-3.225509, -0.217916) -- (-3.234525, -0.179504) -- (-3.241627, -0.142298) -- (-3.246988, -0.106017) -- (-3.250733, -0.070412) -- (-3.252947, -0.035252) -- (-3.253679, -0.000316) -- (-3.252947, 0.034605) -- (-3.250736, 0.069728) -- (-3.246998, 0.105268) -- (-3.241651, 0.141454) -- (-3.234573, 0.178534) -- (-3.225595, 0.216782) -- (-3.214495, 0.256510) -- (-3.200979, 0.298077) -- (-3.184664, 0.341913) -- (-3.165049, 0.388537) -- (-3.141473, 0.438597) -- (-3.113047, 0.492918) -- (-3.078559, 0.552594) -- (-3.036293, 0.619122) -- (-2.983715, 0.694677) -- (-2.916775, 0.782738) -- (-2.825806, 0.892430) -- (-2.079557, 1.708734) -- (-1.966032, 1.824645) -- (-1.844941, 1.938337) -- (-1.329150, 2.380307) -- (-1.066773, 2.586359) -- (-0.891512, 2.715923) -- (-0.800373, 2.776536) -- (-0.717410, 2.826824) -- (-0.640992, 2.868871) -- (-0.570017, 2.904131) -- (-0.503619, 2.933716) -- (-0.441075, 2.958497) -- (-0.381763, 2.979167) -- (-0.325133, 2.996276) -- (-0.270674, 3.010268) -- (-0.217849, 3.021503) -- (-0.165720, 3.030321) -- (-0.104812, 3.037919) -- (-0.045159, 3.043124) -- (-0.003489, 3.044847) -- (0.062811, 3.044673) -- (0.260927, 3.035448) -- (0.801321, 2.989553) -- (1.144291, 2.947185) -- (1.328101, 2.916371) -- (1.461922, 2.888033) -- (1.573364, 2.859476) -- (1.671526, 2.829875) -- (1.760380, 2.798960) -- (1.842057, 2.766642) -- (1.917885, 2.732898) -- (1.988780, 2.697718) -- (2.055423, 2.661092) -- (2.118343, 2.622995) -- (2.177968, 2.583388) -- (2.234648, 2.542213) -- (2.288678, 2.499394) -- (2.340305, 2.454833) -- (2.389741, 2.408410) -- (2.437165, 2.359982) -- (2.482727, 2.309377) -- (2.526558, 2.256391) -- (2.568764, 2.200779) -- (2.609435, 2.142251) -- (2.648647, 2.080448) -- (2.686472, 2.014914) -- (2.722993, 1.945029) -- (2.758356, 1.869837) -- (2.792956, 1.787457) -- (2.828363, 1.692057) -- (2.877185, 1.540281) -- (3.141286, 0.572913) -- (3.228687, 0.214480) -- (3.240522, 0.153215) -- (3.246779, 0.111019) -- (3.250774, 0.073094) -- (3.253063, 0.036779) -- (3.253810, 0.001177) -- cycle;

		\draw [thick] (-1.9094,-1.9094) -- (2.3989,2.3989);
		\node[circle, fill, inner sep=1.5pt] at (-1.9094,-1.9094) {};
		\node[circle, fill, inner sep=1.5pt] at (2.399135,2.399135) {};
		\node[circle, fill, inner sep=1.5pt] at (-3.253679, -0.000316) {};
		\node[circle, fill, inner sep=1.5pt] at (3.253810, 0.001177) {};

		\node at (-1.2, 0.7) {$W(B+iB^\Gamma)$};
		\node[anchor=east, shift={(-0.2,0)}] at (-1.9094,-1.9094) {$W^{1+i}_\textup{min}(B) \approx -1.9094$};
		\node[anchor=west, shift={(0.2,0)}] at (2.399135,2.399135) {$W^{1+i}_\textup{max}(B) \approx 2.3991$};
		\node[anchor=east, shift={(0,0.3)}] at (-3.253679, -0.000316) {$\lambda_\textup{min}(B) \approx -3.2537$};
		\node[anchor=west, shift={(-0.1,0.4)}] at (3.253810, 0.001177) {$\lambda_\textup{max}(B) \approx 3.2538$};
	\end{tikzpicture}
\caption{The numerical range of a random sparse symmetric matrix $B \in M_{500} \otimes M_{500}$.}\label{fig:large_C}
\end{centering}
\end{figure}
\end{example}

It is perhaps worth noting that Theorems~\ref{thm:mu_minmax_from_num_range} and~\ref{thm:W_diag_equals_p} can themselves be interpreted as semidefinite programming relaxations of $\mu_{\textup{max}}(B)$. For example, $W^{1+i}_\textup{max}(B)$ is the optimal value of the following primal/dual pair of semidefinite programs in the variables $c,p \in \R$ and the symmetric matrix variable $X \in M_m \otimes M_n$:
\begin{align}\begin{split}\label{eq:exact_theorem_SDP}
	\begin{matrix}
		\begin{tabular}{r l c r l}
		\multicolumn{2}{c}{\underline{\textup{Primal problem}}} & \quad \quad & \multicolumn{2}{c}{\underline{\textup{Dual problem}}} \\
		\textup{maximize:} & $\tr(BX)$ & \quad \quad & \textup{minimize:} & $c$ \\
		\textup{subject to:} & $\tr(BX) = \tr(BX^\Gamma)$ & \quad \quad \quad & \textup{subject to:} & $pB + (1-p)B^\Gamma \preceq cI$ \\
		\ & $\tr(X) = 1$ & \quad \quad & \ & \\
		\ & $X \succeq O$ & \quad \quad & \ &	
		\end{tabular}
	\end{matrix}
\end{split}\end{align}
However, it is more computationally efficient to compute $W^{1+i}_\textup{max}(B)$ directly via the numerical range (as described in Appendix~\ref{sec:implementation}), rather than via semidefinite program solvers.

\section{Some Applications}\label{sec:applications}

To illustrate the effectiveness of our results, we now apply Theorems~\ref{thm:mu_minmax_from_num_range} and~\ref{thm:W_diag_equals_p} to some specific problems that arise as special cases of the optimization problem~\eqref{eq:main_minmax}.

\subsection{Application to Rank-One-Avoiding Matrix Subspaces}\label{sec:matrix_subspaces}

We say that a subspace of matrices $\S \subseteq M_{m,n}$ is \emph{rank-one-avoiding} if $\rank(X) \neq 1$ for all $X \in \S$. Although there are a few results on bounding which subspaces can be rank-one-avoiding, for instance those of dimension greater than $mn-\textup{max}\{m,n\}$ cannot be rank-one avoiding~\cite{DGMS10},
showing that a given matrix subspace $\S$ is rank-one-avoiding seems to be hard (indeed, determining the minimal rank of a non-zero matrix in a subspace is NP-hard \cite{BFS99}). However, our method can be adapted to solve this problem quickly for a large number of matrix subspaces.

Consider the following quantity, which can be thought of as a measure of how close $\S$ is to being rank-one-avoiding:
\begin{align}\label{eq:d_S_defn}
    d(\S) \defeq \max_{Y \in \S}\Big\{ \|Y\| : \|Y\|_{\textup{F}} \leq 1 \Big\},
\end{align}
where $\|Y\|$ is the operator norm (i.e., the largest singular value of $Y$) and $\|Y\|_{\textup{F}}$ is the Frobenius norm (i.e., the $2$-norm of the vector of singular values of $Y$). We note that this maximum really is a maximum (not a supremum) by compactness of the set of $Y \in \S$ with $\|Y\|_{\textup{F}} \leq 1$. It is straightforward to show that $\|Y\| \leq \|Y\|_{\textup{F}}$, with equality if and only if $\rank(Y) = 1$, which immediately implies that $d(\S) \leq 1$ with equality if and only if there exists $Y \in \S$ with $\rank(Y) = 1$. In other words, $\S$ is rank-one-avoiding if and only if $d(\S) < 1$.

Before we can state the main result of this section, we recall that the (column-by-column) vectorization of a matrix is the operation $\mathrm{vec} : M_{m,n} \to \R^n \otimes \R^m$ defined by
\begin{align*}
    \mathrm{vec}\left(\left[\begin{array}{c|c|c|c}
    \mathbf{v_1} & \mathbf{v_2} & \cdots & \mathbf{v_n}\end{array}\right]\right) \defeq
\left[\begin{array}{c}
    \mathbf{v_1} \\ \hline
    \mathbf{v_2} \\ \hline
    \vdots \\ \hline
    \mathbf{v_n}\end{array}\right].
\end{align*}
For a subspace $\S \subseteq M_{m,n}$, we similarly define $\mathrm{vec}(\S) \defeq \{\mathrm{vec}(Y) : Y \in \S\}$.

Our main result of this section is then an easy-to-compute upper bound on $d(\S)$ in terms of $\mathrm{vec}(\S)$, and thus a way to show that a subspace $\S$ is rank-one-avoiding:

\begin{theorem}\label{thm:rank_one_avoiding_by_num_range}
    Let $\S \subseteq M_{m,n}$ be a subspace and let $P_{\S}$ be the orthogonal projection onto $\mathrm{vec}(\S)$. Then
    \[
        d(\S)^2 \leq W_{\textup{max}}^{1+i}(P_{\S}).
    \]
    In particular, if $W_{\textup{max}}^{1+i}(P_{\S}) < 1$ then $\S$ is rank-one-avoiding.
\end{theorem}

\begin{proof}
    By Theorem~\ref{thm:mu_minmax_from_num_range}, we know that $\mu_{\textup{max}}(P_{\S}) \leq W_{\textup{max}}^{1+i}(P_{\S})$, so it suffices to show that $\mu_{\textup{max}}(P_{\S}) = d(\S)^2$. To this end, we compute
    \begin{align*}
        d(\S)^2 & = \max_{Y \in \S}\Big\{ \|Y\|^2 : \|Y\|_{\textup{F}} \leq 1 \Big\} \\
        & = \max_{\substack{Y \in \S\\\v \in \R^m,\w \in \R^n}}\Big\{ |\w^TY\v|^2 : \|Y\|_{\textup{F}}, \|\v\|, \|\w\| \leq 1 \Big\} \\
        & = \max_{\substack{Y \in \S\\\v \in \R^m,\w \in \R^n}}\Big\{ |\mathrm{vec}(Y)^T(\v \otimes \w)|^2 : \|\mathrm{vec}(Y)\|, \|\v\|, \|\w\| \leq 1 \Big\} \\
        & = \max_{\v \in \R^m,\w \in \R^n}\Big\{ (\v \otimes \w)^TP_{\S}(\v \otimes \w) : \|\v\|, \|\w\| \leq 1 \Big\} \\
        & = \mu_{\textup{max}}(P_{\S}),
    \end{align*}
    where the third equality comes from the facts that $\|Y\|_{\textup{F}} = \|\mathrm{vec}(Y)\|$ and $\w^TY\v = \mathrm{vec}(Y)^T(\v \otimes \w)$.
\end{proof}

\begin{example}\label{exam:2dim_real_subspace}
    Consider the subspace $\S \subset M_2$ of matrices of the form
    \begin{align}\label{eq:22subspace}
        \begin{bmatrix}
            c & -d \\ d & c
        \end{bmatrix},
    \end{align}
    where $c,d \in \R$. This subspace is rank-one-avoiding since for all non-zero matrices in $\S$, the determinant $c^2+d^2$ is non-zero, so each matrix is invertible and thus has rank $2$ (as a side note, $\S$ is a largest possible rank-one-avoiding subspace of $M_2$, since its dimension ($2$) saturates the upper bound of $mn-\textup{max}(m,n)=2$ from~\cite{DGMS10}).
    
    To see that $\S$ is rank-one-avoiding via our machinery, we note that direct computation shows that
    \[
        P_{\S} = \frac{1}{2}\begin{bmatrix}
            1 & 0 & 0 & 1 \\
            0 & 1 & -1 & 0 \\
            0 & -1 & 1 & 0 \\
            1 & 0 & 0 & 1
        \end{bmatrix}
    \]
    and that $W_{\textup{max}}^{1+i}(P_{\S}) = 1/2$. It then follows from Theorem~\ref{thm:rank_one_avoiding_by_num_range} that $\S$ is rank-one-avoiding.
\end{example}

\begin{remark}
    The previous example demonstrates that our results really are specific to the case of real matrices; the subspace $\S$ from Example~\ref{exam:2dim_real_subspace} is rank-one avoiding in the real case, but not in the complex case (e.g., if we choose $c = 1$ and $d = i$ then the matrix~\eqref{eq:22subspace} has rank~$1$).\footnote{In fact, every $2$-dimensional subspace of $M_2(\C)$ contains a rank-$1$ matrix \cite{CMW08}.} Our method is thus independent of known methods for the showing that a subspace of complex matrices is rank-one-avoiding (see \cite{DRA21,JLV22}, for example).
\end{remark}

Since we can easily compute $W^{1+i}_\textup{max}(P_\S)$ for any subspace $\S$, we can get a rough idea of the efficacy of Theorem~\ref{thm:rank_one_avoiding_by_num_range} by generating orthogonal projection matrices for random subspaces of a given dimension, and computing $W^{1+i}_\textup{max}(P_\S)$ for each. Figure~\ref{fig:rnk-1-avoid-prob} illustrates the probability of Theorem~\ref{thm:rank_one_avoiding_by_num_range} detecting that a subspace of the given dimension is rank-one-avoiding, based on $10,000$ randomly-generated (according to Haar measure on the Grassmannian) subspaces of each dimension.



\begin{figure}[htb]
\centering
\begin{subfigure}[b]{0.53\textwidth}
    \centering
    \includegraphics[width=\textwidth]{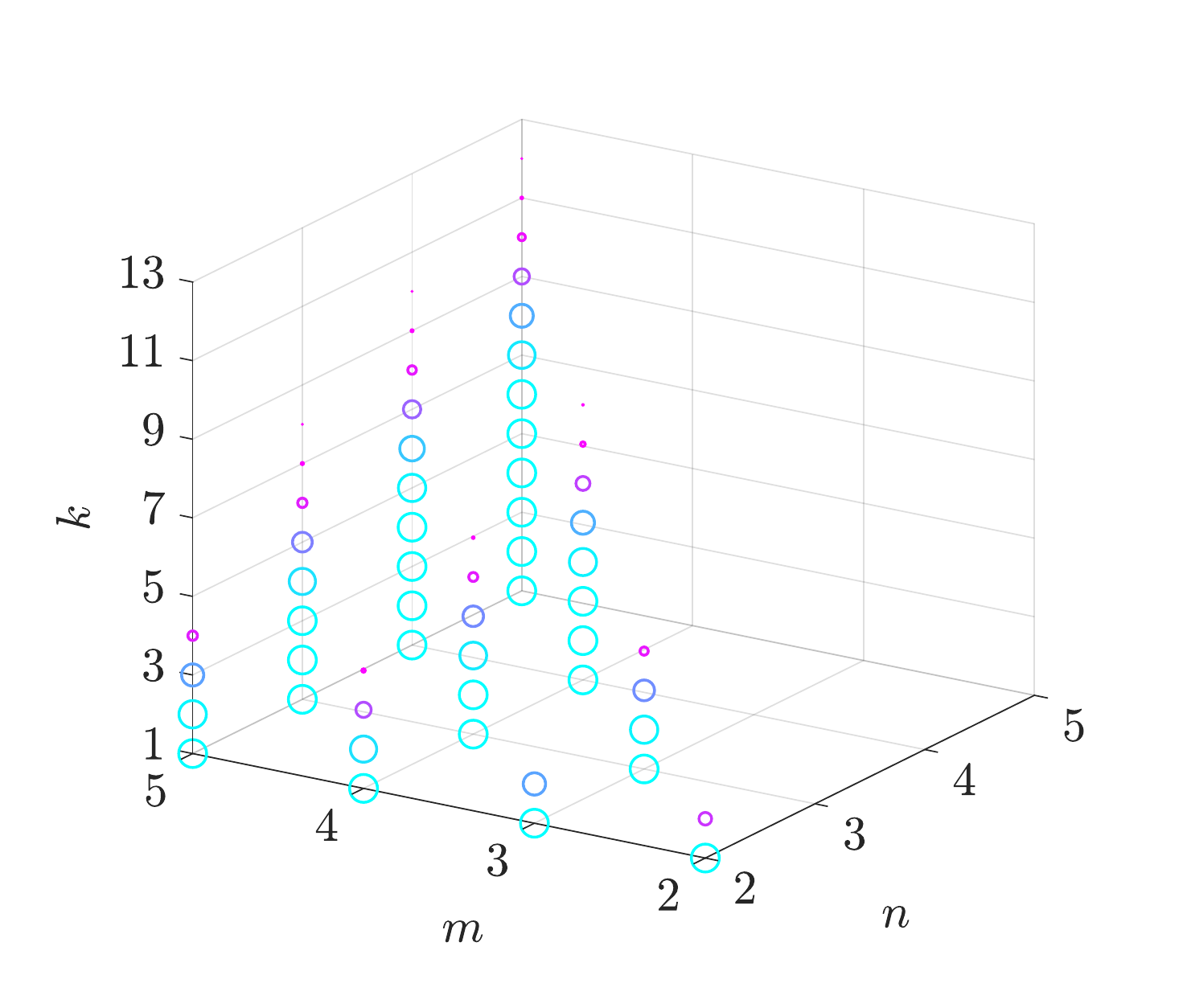}
    \caption{A visualization of the probabilities when $m \geq n$ (the picture is symmetric when $m \leq n$), where the largest circles shown mean probability $1$ and no circle means probability $0$.}
\end{subfigure} \hfill \begin{subfigure}[b]{0.43\textwidth}
\centering
    \begin{tabular}{c|cccc}
    \diagbox{$k$}{$n$} & $2$ & $3$ & $4$ & $5$ \\ \hline
    $1$  & $1.00$     & $1.00$     & $1.00$     & $1.00$     \\
    $2$  & $0.21$  & $0.96$  & $1.00$     & $1.00$     \\
    $3$  & $0.00$     & $0.56$  & $1.00$     & $1.00$     \\
    $4$  & $\cdot$     & $0.09$  & $0.95$  & $1.00$     \\
    $5$  & $\cdot$ & $0.00$     & $0.69$  & $1.00$     \\
    $6$  & $\cdot$ & $\cdot$     & $0.25$  & $0.99$  \\
    $7$  & $\cdot$ & $\cdot$     & $0.03$  & $0.92$  \\
    $8$  & $\cdot$ & $\cdot$     & $0.00$     & $0.68$  \\
    $9$  & $\cdot$ & $\cdot$     & $\cdot$     & $0.31$  \\
    $10$ & $\cdot$ & $\cdot$ & $\cdot$     & $0.07$  \\
    $11$ & $\cdot$ & $\cdot$ & $\cdot$     & $0.00$     \\
    $12$ & $\cdot$ & $\cdot$ & $\cdot$     & $\cdot$     \\
    \end{tabular}
    \caption{Explicit probabilities, rounded to two decimal places, when $m = n$.}
\end{subfigure}
\caption{The approximate probability of Theorem~\ref{thm:rank_one_avoiding_by_num_range} detecting that a $k$-dimensional subspace of $M_{m,n}$ is rank-one-avoiding, based on $10,000$ randomly-generated (according to Haar measure on the Grassmannian) subspaces of each dimension.}\label{fig:rnk-1-avoid-prob}
\end{figure}


\subsection{Application to Positive Linear Maps and Biquadratic Forms}\label{sec:pos_lin_map}

Recall that we use $M_n^\textup{S}$ to denote the set of (real) symmetric $n \times n$ matrices. We similarly use $M_n^\textup{A}$ to denote the set of (real) skew-symmetric $n \times n$ matrices (i.e., matrices $Y$ satisfying $Y^T = -Y$). We say that a linear map $\Phi : M_m \rightarrow M_n$ is \emph{transpose-preserving} if $\Phi(X^T) = \Phi(X)^T$ for all $X \in M_m$ (or equivalently, if $\Phi(M_m^\textup{S}) \subseteq M_n^\textup{S}$ and $\Phi(M_m^\textup{A}) \subseteq M_n^\textup{A}$).\footnote{These maps were called \emph{Hermitian} in \cite{CDPR22}; we instead use the terminology of \cite{JohALA}.} If $\Phi$ is transpose-preserving then we say that it is \emph{positive} if $\Phi(X) \in M_n^\textup{S}$ is positive semidefinite whenever $X \in M_m^\textup{S}$ is positive semidefinite \cite{Pau03,CDPR22}. It is straightforward to show that if there exist matrices $\{A_j\} \subseteq M_{n,m}$ and a linear map $\Psi : M_m \rightarrow M_n$ such that $\Psi(X) = O$ whenever $X \in M_m^\textup{S}$ and
\begin{align}\label{eq:decomp_real}
    \Phi(X) = \sum_j A_j X A_j^T + \Psi(X) \quad \text{for all} \quad X \in M_m,
\end{align}
then $\Phi$ is positive (in fact, maps of the form $X \mapsto \sum_j A_j X A_j^T$ are called \emph{completely positive}).

Conversely, if $\min\{m,n\} \leq 2$ then it is known \cite{Cal73} that if $\Phi$ is positive then it can be written in the form~\eqref{eq:decomp_real}. However, if $\min\{m,n\} \geq 3$ then there are exceptions (i.e., positive linear maps that cannot be written in the form~\eqref{eq:decomp_real}). For example, the \emph{Choi map} $\Phi : M_3 \rightarrow M_3$ defined by
\begin{align}\label{eq:choi_map_3x3}
    \Phi(X) := \begin{bmatrix}
        x_{1,1} + x_{3,3} & -x_{1,2} & -x_{1,3} \\
        -x_{2,1} & x_{1,1} + x_{2,2} & -x_{2,3} \\
        -x_{3,1} & -x_{3,2} & x_{2,2} + x_{3,3}
    \end{bmatrix}
\end{align}
is positive but cannot be written in the form~\eqref{eq:decomp_real} \cite{Cho75}.

The \emph{Choi matrix} \cite{Cho75b} of a linear map $\Phi : M_m \rightarrow M_n$ is the matrix $C_\Phi \in M_m \otimes M_n$ defined by
\[
    C_{\Phi} \defeq \sum_{i,j=1}^m E_{i,j} \otimes \Phi(E_{i,j}),
\]
where $E_{i,j}$ is the matrix with a $1$ in its $(i,j)$-entry and zeroes elsewhere. It is straightforward to show that $\Phi$ is transpose-preserving if and only if $C_\Phi$ is symmetric. Furthermore, a transpose-preserving map $\Phi$ can be written in the form~\eqref{eq:decomp_real} if and only if $C_{\Phi} = A + B$ for some symmetric $A, B \in M_m \otimes M_n$ with $A$ positive semidefinite and $B$ satisfying $B = -B^\Gamma$ ($A$ and $B$ are the Choi matrices of the maps $X \mapsto \sum_j A_jXA_j^T$ and $\Psi$, respectively, and it is well-known \cite{Cho75b} that a map is completely positive if and only if its Choi matrix is positive semidefinite).

In general, determining positivity of a linear map is hard, since it is equivalent to solving the optimization problem~\eqref{eq:main_minmax}:

\begin{theorem}\label{thm:positive_map}
    Let $\Phi : M_m \rightarrow M_n$ be a transpose-preserving linear map. Then
    \begin{align}\label{eq:pos_map_as_min}
        \min\big\{ \lambda_{\textup{min}}\big(\Phi(X)\big) : X \succeq O, \tr(X) = 1 \big\} = \mu_{\textup{min}}(C_{\Phi}).
    \end{align}
    In particular, $\Phi$ is positive if and only if $\mu_{\textup{min}}(C_{\Phi}) \geq 0$.
\end{theorem}

\begin{proof}
    By convexity and the spectral decomposition, $\lambda_{\textup{min}}\big(\Phi(X)\big)$ is minimized as in Equation~\eqref{eq:pos_map_as_min} when $X = \x\x^T$ for some $\x \in \R^m$ with $\|\x\| = 1$. It follows that
    \begin{align*}
        \min\big\{ \lambda_{\textup{min}}\big(\Phi(X)\big) : X \succeq O, \tr(X) = 1 \big\} & = \min_{\x \in \R^m}\big\{ \lambda_{\textup{min}}\big(\Phi(\x\x^T)\big) : \|\x\| = 1 \big\} \\
         & = \min_{\x \in \R^m, \y \in \R^n}\big\{ \y^T\Phi(\x\x^T)\y : \|\x\| = \|\y\| = 1 \big\} \\
         & = \min_{\x \in \R^m, \y \in \R^n}\big\{ (\x \otimes \y)^T C_{\Phi}(\x \otimes \y) : \|\x\| = \|\y\| = 1 \big\} \\
         & = \mu_{\textup{min}}(C_{\Phi}),
    \end{align*}
    as claimed.
\end{proof}

Determining whether or not $\Phi$ can be written in the form~\eqref{eq:decomp_real} is much simpler than determining positivity, as it can be decided by the following semidefinite program in the variables $c \in \R$ and $X,Y \in M_m \otimes M_n$:
\begin{align}\begin{split}\label{eq:sdp_decomp_real}
    \textup{minimize:} & \ c \\
    \textup{subject to:} & \ C_{\Phi} + cI = X + Y \\
    & \ Y^\Gamma = -Y \\
    & \ X \succeq O.
\end{split}\end{align}
In particular, such a decomposition of $\Phi$ exists if and only if the optimal value of this semidefinite program is non-positive.

However, as with the examples from Section~\ref{sec:comparison_with_known}, this semidefinite program is only feasible to run when $m,n < 20$ or so. The numerical range approach of Theorem~\ref{thm:mu_minmax_from_num_range} can be adapted to this problem for much larger values of $m$ and $n$:

\begin{theorem}\label{thm:decomp_pos_map_from_range}
    Let $\Phi : M_m \rightarrow M_n$ be a transpose-preserving linear map. If $W^{1+i}_{\textup{min}}(C_\Phi) \geq 0$ then $\Phi$ can be written in the form~\eqref{eq:decomp_real} and is thus positive.
\end{theorem}

\begin{proof}
    Since $\Phi$ is transpose-preserving, $C_\Phi$ is symmetric, so know from Theorem~\ref{thm:W_diag_equals_p} that
    \[
        W^{1+i}_{\textup{min}}(C_\Phi) = \max_{p \in \R}\big\{\lambda_{\textup{min}}(pC_{\Phi} + (1-p)C_{\Phi}^\Gamma)\big\}.
    \]
    The hypothesis $W^{1+i}_{\textup{min}}(C_\Phi) \geq 0$ is thus equivalent to the existence of a particular $p \in \R$ such that $pC_{\Phi} + (1-p)C_{\Phi}^\Gamma$ is positive semidefinite. That is, $\Omega := p\Phi + (1-p)(\Phi \circ T)$ is completely positive, where $T$ is the transpose map.
    
    Rearranging this equation shows that $\Phi = \Omega - (1-p)(\Phi \circ T - \Phi)$. Since $\Omega$ is completely positive, it can be written in the form $\Omega(X) = \sum_j A_jXA_j^T$, and we can choose $\Psi := -(1-p)(\Phi \circ T - \Phi)$ so that $\Phi = \Omega + \Psi$ is a decomposition of the form~\eqref{eq:decomp_real}.
\end{proof}

Despite how much computationally simpler Theorem~\ref{thm:decomp_pos_map_from_range} is to implement than the semidefinite program~\eqref{eq:sdp_decomp_real}, it performs very well in practice, as illustrated by the next example.

\begin{example}\label{exx:choi_map}
    Let $c \in \R$ and consider the linear map $\Phi_c : M_3 \rightarrow M_3$ defined by
    \begin{align}\label{eq:generalized_choi_map}
        \Phi_c(X) := \begin{bmatrix}
            x_{1,1} + cx_{2,2} + x_{3,3} & -x_{1,2} & -x_{1,3} \\
            -x_{2,1} & x_{1,1} + x_{2,2} + cx_{3,3} & -x_{2,3} \\
            -x_{3,1} & -x_{3,2} & cx_{1,1} + x_{2,2} + x_{3,3}
        \end{bmatrix}.
    \end{align}
    This map is positive exactly when $c \geq 0$ (when $c = 0$ it is the Choi map from Equation~\eqref{eq:choi_map_3x3}), and the semidefinite program~\eqref{eq:sdp_decomp_real} shows that it can be written in the form~\eqref{eq:decomp_real} exactly when $c \geq 1/4$.
    
    Theorem~\ref{thm:decomp_pos_map_from_range} captures this behavior perfectly: $W^{1+i}_{\textup{min}}(C_{\Phi_c}) \geq 0$ exactly when $c \geq 1/4$ (see Figure~\ref{fig:choi_like}). By contrast, the trivial eigenvalue bounds $\lambda_{\textup{min}}(C_{\Phi_c})$ and $\lambda_{\textup{min}}(C_{\Phi_c}^\Gamma)$ only show that $\Phi_c$ can be written in the form~\eqref{eq:decomp_real} when $c \geq 1$, since $\lambda_{\textup{min}}(C_{\Phi_c}) = -1 < 0$ for all $c$ and $\lambda_{\textup{min}}(C_{\Phi_c}^\Gamma) \geq 0$ only when $c \geq 1$.

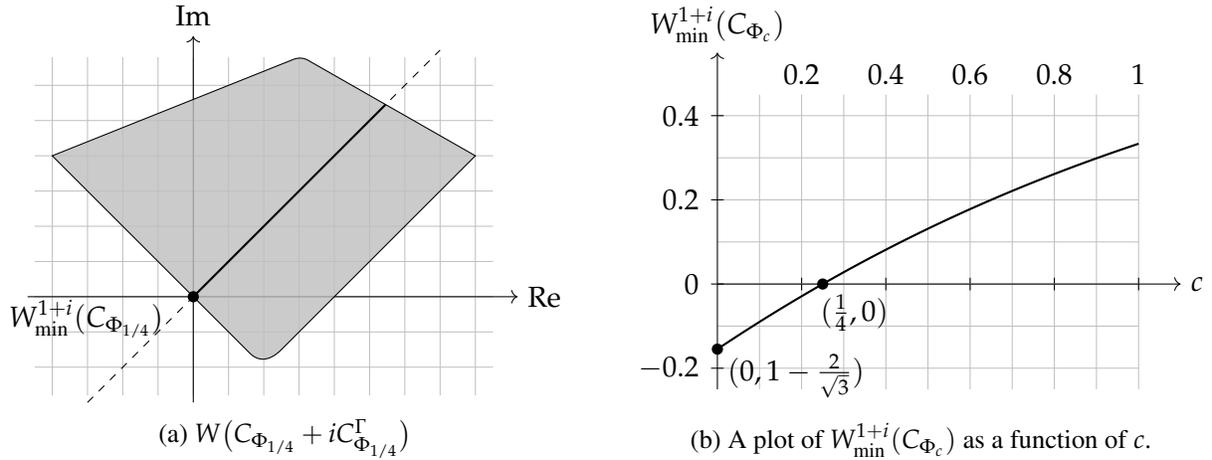
\begin{figure}[htb]
\centering
\begin{subfigure}[b]{0.48\textwidth}
    \centering
	\begin{tikzpicture}[scale=1.875]
		\foreach \x in {-1,-0.75,...,2} {
			\draw[ultra thin, gray!50!white] (\x, -0.7) -- (\x, 1.7);
		}

		\foreach \y in {-0.5,-0.25,...,1.5} {
			\draw[ultra thin, gray!50!white] (-1.125, \y) -- (2.125, \y);
		}

		\draw [->] (-1.175, 0) -- (2.3, 0) node[anchor=west] {$\mathrm{Re}$}; 
		\draw [->] (0, -0.75) -- (0, 1.85) node[anchor=south] {$\mathrm{Im}$}; 

		\draw [dashed] (-0.75,-0.75) -- (1.75,1.75); 

		\draw [fill=gray!50, fill opacity=0.8]
(2.000000, 1.000000) -- (0.622085, -0.377884) -- (0.611004, -0.388299) -- (0.600760, -0.397147) -- (0.591252, -0.404689) -- (0.582389, -0.411133) -- (0.574093, -0.416649) -- (0.566297, -0.421374) -- (0.558941, -0.425420) -- (0.551973, -0.428881) -- (0.545348, -0.431833) -- (0.539025, -0.434338) -- (0.532970, -0.436448) -- (0.527149, -0.438206) -- (0.521536, -0.439649) -- (0.516104, -0.440804) -- (0.510830, -0.441697) -- (0.505694, -0.442346) -- (0.500675, -0.442768) -- (0.495756, -0.442975) -- (0.490920, -0.442975) -- (0.486151, -0.442775) -- (0.481434, -0.442379) -- (0.476755, -0.441788) -- (0.472100, -0.441001) -- (0.467455, -0.440013) -- (0.462807, -0.438820) -- (0.458143, -0.437411) -- (0.453448, -0.435776) -- (0.448709, -0.433899) -- (0.443912, -0.431762) -- (0.439042, -0.429344) -- (0.434083, -0.426617) -- (0.429020, -0.423550) -- (0.423836, -0.420104) -- (0.418512, -0.416234) -- (0.413029, -0.411886) -- (0.407366, -0.406996) -- (0.401500, -0.401485) -- (-1.000000, 1.000000) -- (0.710974, 1.684338) -- (0.717029, 1.686448) -- (0.722850, 1.688206) -- (0.728463, 1.689649) -- (0.733895, 1.690804) -- (0.739169, 1.691697) -- (0.744305, 1.692346) -- (0.749324, 1.692768) -- (0.754243, 1.692975) -- (0.759079, 1.692975) -- (0.763848, 1.692775) -- (0.768565, 1.692379) -- (0.773244, 1.691788) -- (0.777899, 1.691001) -- (0.782544, 1.690013) -- (0.787192, 1.688820) -- (0.791856, 1.687411) -- (0.796551, 1.685776) -- (0.801290, 1.683899) -- (0.806087, 1.681762) -- (0.810957, 1.679344) -- cycle;

		\draw [thick] (0,0) -- (1.363636,1.363636);
		\node[circle, fill, inner sep=1.5pt] at (0,0) {};

		\node[anchor=north east, shift={(-0.26,0.09)}] at (0,0) {$W^{1+i}_\textup{min}(C_{\Phi_{1/4}})$};
	\end{tikzpicture}
    \caption{$W\big(C_{\Phi_{1/4}} + iC_{\Phi_{1/4}}^\Gamma\big)$}\label{fig:WA_Choi_map}
\end{subfigure} \hfill \begin{subfigure}[b]{0.48\textwidth}
    \centering
	\begin{tikzpicture}[scale=5.6]
		\foreach \x in {0.1,0.2,...,1} {
			\draw[ultra thin, gray!50!white] (\x, -0.25) -- (\x, 0.45);
		}
		\foreach \x in {0.2,0.4,...,1} {
			\draw (\x, -0.015) -- (\x, 0.015);
		}
		\node[anchor=south] at (0.2, 0.45) {$0.2$};
		\node[anchor=south] at (0.4, 0.45) {$0.4$};
		\node[anchor=south] at (0.6, 0.45) {$0.6$};
		\node[anchor=south] at (0.8, 0.45) {$0.8$};
		\node[anchor=south] at (1, 0.45) {$1$};

		\foreach \y in {-0.2,-0.1,...,0.4} {
			\draw[ultra thin, gray!50!white] (0, \y) -- (1, \y);
		}
		\foreach \y in {-0.2,0,...,0.4} {
			\draw (-0.015, \y) -- (0.015, \y);
		}
		\node[anchor=east] at (-0.02, -0.2) {$-0.2$};
		\node[anchor=east] at (-0.02, 0) {$0$};
		\node[anchor=east] at (-0.02, 0.2) {$0.2$};
		\node[anchor=east] at (-0.02, 0.4) {$0.4$};

		\draw[->] (0, 0) -- (1.1, 0) node[anchor=west] {$c$}; 
		\draw[->] (0, -0.25) -- (0, 0.55) node[anchor=south] {$W^{1+i}_{\textup{min}}(C_{\Phi_c})$}; 

		\draw[thick]
(0.000000, -0.154700) -- (0.010000, -0.148053) -- (0.020000, -0.141444) -- (0.030000, -0.134873) -- (0.040000, -0.128341) -- (0.050000, -0.121848) -- (0.060000, -0.115393) -- (0.070000, -0.108976) -- (0.080000, -0.102598) -- (0.090000, -0.096258) -- (0.100000, -0.089956) -- (0.110000, -0.083693) -- (0.120000, -0.077468) -- (0.130000, -0.071281) -- (0.140000, -0.065133) -- (0.150000, -0.059022) -- (0.160000, -0.052950) -- (0.170000, -0.046915) -- (0.180000, -0.040919) -- (0.190000, -0.034960) -- (0.200000, -0.029039) -- (0.210000, -0.023156) -- (0.220000, -0.017311) -- (0.230000, -0.011503) -- (0.240000, -0.005732) -- (0.250000, 0.000000) -- (0.260000, 0.005695) -- (0.270000, 0.011354) -- (0.280000, 0.016975) -- (0.290000, 0.022559) -- (0.300000, 0.028106) -- (0.310000, 0.033617) -- (0.320000, 0.039091) -- (0.330000, 0.044528) -- (0.340000, 0.049929) -- (0.350000, 0.055293) -- (0.360000, 0.060621) -- (0.370000, 0.065913) -- (0.380000, 0.071169) -- (0.390000, 0.076389) -- (0.400000, 0.081574) -- (0.410000, 0.086722) -- (0.420000, 0.091836) -- (0.430000, 0.096914) -- (0.440000, 0.101956) -- (0.450000, 0.106964) -- (0.460000, 0.111937) -- (0.470000, 0.116875) -- (0.480000, 0.121779) -- (0.490000, 0.126648) -- (0.500000, 0.131482) -- (0.510000, 0.136283) -- (0.520000, 0.141050) -- (0.530000, 0.145783) -- (0.540000, 0.150482) -- (0.550000, 0.155147) -- (0.560000, 0.159780) -- (0.570000, 0.164379) -- (0.580000, 0.168945) -- (0.590000, 0.173479) -- (0.600000, 0.177979) -- (0.610000, 0.182448) -- (0.620000, 0.186884) -- (0.630000, 0.191287) -- (0.640000, 0.195659) -- (0.650000, 0.199999) -- (0.660000, 0.204308) -- (0.670000, 0.208585) -- (0.680000, 0.212831) -- (0.690000, 0.217046) -- (0.700000, 0.221230) -- (0.710000, 0.225383) -- (0.720000, 0.229506) -- (0.730000, 0.233599) -- (0.740000, 0.237661) -- (0.750000, 0.241694) -- (0.760000, 0.245697) -- (0.770000, 0.249670) -- (0.780000, 0.253614) -- (0.790000, 0.257528) -- (0.800000, 0.261414) -- (0.810000, 0.265271) -- (0.820000, 0.269099) -- (0.830000, 0.272899) -- (0.840000, 0.276670) -- (0.850000, 0.280414) -- (0.860000, 0.284130) -- (0.870000, 0.287818) -- (0.880000, 0.291478) -- (0.890000, 0.295112) -- (0.900000, 0.298718) -- (0.910000, 0.302297) -- (0.920000, 0.305850) -- (0.930000, 0.309376) -- (0.940000, 0.312876) -- (0.950000, 0.316350) -- (0.960000, 0.319797) -- (0.970000, 0.323219) -- (0.980000, 0.326616) -- (0.990000, 0.329987) -- (1.000000, 0.333333);

		\node[circle, fill, inner sep=1.5pt] at (0.25,0) {};
		\node[circle, fill, inner sep=1.5pt] at (0,-0.1547) {};

		\node[anchor=north west] at (0.22,0) {$(\tfrac{1}{4},0)$};
		\node[anchor=north west, shift={(0,0.07)}] at (0,-0.1547) {$(0,1-\tfrac{2}{\sqrt{3}})$};
	\end{tikzpicture}
	
    \caption{A plot of $W^{1+i}_{\textup{min}}(C_{\Phi_{c}})$ as a function of $c$.}
\end{subfigure}
\caption{Theorem~\ref{thm:decomp_pos_map_from_range} shows that the map $\Phi_c$ from Equation~\eqref{eq:generalized_choi_map} can be written in the form~\eqref{eq:decomp_real}, and is thus positive, when $c \geq 1/4$.}\label{fig:choi_like}
\end{figure}
\end{example}

Since there is a straightforward correspondence between positive linear maps and biquadratic forms (see \cite{Cho75} for details), the results in this section immediately apply to the problem of trying to show that a biquadratic form can be written as a sum of squares.

\section{Generalization to Complex Matrices}\label{sec:complex}

Given the ubiquity of complex tensor optimizations and positive linear maps acting on complex matrices in quantum information theory (see \cite{Joh12,HM13} and the references therein, for example), it is natural to ask how far our results from the previous section can generalize to optimizations over complex product vectors. We thus now define complex versions of the quantities from Equation~\eqref{eq:main_minmax}:
\begin{align}\begin{split}\label{eq:main_minmax_complex}
    \mu_{\textup{min}}^\C(B) & \defeq \min_{\mathbf{v} \in \C^m,\mathbf{w} \in \C^n}\big\{ (\mathbf{v} \otimes \mathbf{w})^*B(\mathbf{v} \otimes \mathbf{w}) : \|\mathbf{v}\| = \|\mathbf{w}\| = 1 \big\} \quad \text{and} \\
    \mu_{\textup{max}}^\C(B) & \defeq \max_{\v \in \C^m,\w \in \C^n}\big\{ (\v \otimes \mathbf{w})^*B(\v \otimes \w) : \|\v\| = \|\w\| = 1 \big\}.
\end{split}\end{align}

Note that we still require $B$ to be real symmetric (not complex Hermitian), since our goal is to find cases where the optimal values of the complex optimization problems from Equation~\eqref{eq:main_minmax_complex} simply equal the corresponding real optimal values from Equation~\eqref{eq:main_minmax}. If we were to allow $B$ to be complex Hermitian then there is essentially no hope of this happening, as illustrated by the next example:

\begin{example}\label{exx:real_opt_required_even_hermit}
    Consider the Hermitian matrix
    \[
        B = \begin{bmatrix}
            0 & 0 & 0 & i \\
            0 & 0 & i & 0 \\
            0 & -i & 0 & 0 \\
            -i & 0 & 0 & 0
        \end{bmatrix} \in M_2(\C) \otimes M_2(\C).
    \]
    It is straightforward to check that, for all $\v, \w \in \R^2$, we have
    \[
        (\v \otimes \w)^*B(\v \otimes \w) =
        \begin{bmatrix}
            v_1w_1 & v_1w_2 & v_2w_1 & v_2w_2
        \end{bmatrix}
        \begin{bmatrix}
            iv_2w_2 \\ iv_2w_1 \\ -iv_1w_2 \\ -iv_1w_1
        \end{bmatrix}
        = 0,
    \]
    which implies $\mu_{\textup{min}}(B) = \mu_{\textup{max}}(B) = 0$.\footnote{Strictly speaking, $\mu_{\textup{min}}(B)$ and $\mu_{\textup{max}}(B)$ are not defined here, since $B$ was assumed to be real in the defining Equation~\eqref{eq:main_minmax}. For the purposes of this example we have just extended the definition to allow complex matrices $B$ (while still optimizing over real vectors $\v$ and $\w$).} On the other hand, $\mu_{\textup{min}}^{\C}(B) = -1$ and $\mu_{\textup{max}}^{\C}(B) = 1$, which can be verified by noting that the eigenvalues of $B$ are $\lambda_{\textup{max}}(B) = 1$ and $\lambda_{\textup{min}}(B) = -1$, and some corresponding eigenvectors are $(1,-i) \otimes (1,1)$ and $(1,i) \otimes (1,1)$, respectively.
\end{example}

The following lemma shows that if we require $B$ to be real and have some extra symmetry properties, then problems like the one described in Example~\ref{exx:real_opt_required_even_hermit} no longer occur, and we do indeed have $\mu_{\textup{min}}(B) = \mu_{\textup{min}}^{\C}(B)$ and $\mu_{\textup{max}}(B) = \mu_{\textup{max}}^{\C}(B)$:

\begin{lemma}\label{lem:bisymmetric_part_hermitian}
    Let $B \in M_m \otimes M_n$ be symmetric and define $X = (B + B^\Gamma)/2$. Then
    \[
        \mu_{\textup{min}}(B) = \mu_{\textup{min}}(X) = \mu_{\textup{min}}^{\C}(X) \quad \text{and} \quad \mu_{\textup{max}}(B) = \mu_{\textup{max}}(X) = \mu_{\textup{max}}^{\C}(X).
    \]
\end{lemma}

\begin{proof}
    Since $B$ is symmetric, we know from Lemma~\ref{lem:bisymmetric_part_real} that
    \[
        \mu_{\textup{min}}(B) = \mu_{\textup{min}}(X) \quad \text{and} \quad \mu_{\textup{max}}(B) = \mu_{\textup{max}}(X).
    \]
    All that remains is to show that $\mu_{\textup{min}}(X) = \mu_{\textup{min}}^{\C}(X)$ and $\mu_{\textup{max}}(X) = \mu_{\textup{max}}^{\C}(X)$. We just prove that $\mu_{\textup{max}}(X) = \mu_{\textup{max}}^{\C}(X)$, since the corresponding minimization result follows similarly. Furthermore, the inequality $\mu_{\textup{max}}(X) \leq \mu_{\textup{max}}^{\C}(X)$ is immediate from the definition, so we just prove the opposite inequality.
    
    Thanks to compactness of the set of unit product vectors, there exist $\v \in \C^m$ and $\w \in \C^n$ with $\|\v\| = \|\w\| = 1$ and $(\v \otimes \w)^* X(\v \otimes \w) = \mu_{\textup{max}}^{\C}(X)$. If we use the fact that $X^\Gamma = X$, then algebra similar to that of Equation~\eqref{eq:vw_part_trans} then shows that
    \begin{align*}
        \tr\big(X(\mathbf{v}\mathbf{v}^* \otimes \mathbf{w}\mathbf{w}^*)\big) = (\mathbf{v} \otimes \mathbf{w})^* X (\mathbf{v} \otimes \mathbf{w}) = (\mathbf{v} \otimes \mathbf{w})^* X^\Gamma(\mathbf{v} \otimes \mathbf{w}) = \tr\big(X(\mathbf{v}\mathbf{v}^* \otimes \overline{\mathbf{w}\mathbf{w}^*})\big),
    \end{align*}
    so
    \begin{align*}
        \mu_{\textup{max}}^{\C}(X) & = \tr\big(X(\mathbf{v}\mathbf{v}^* \otimes \mathrm{Re}(\mathbf{w}\mathbf{w}^*))\big).
    \end{align*}
    Similar algebra, using the fact that $X^T = X$, then shows that
    \begin{align}\label{eq:mu_max_conv}
        \mu_{\textup{max}}^{\C}(X) & = \tr\big(X(\mathrm{Re}(\mathbf{v}\mathbf{v}^*) \otimes \mathrm{Re}(\mathbf{w}\mathbf{w}^*))\big).
    \end{align}
    
    The matrix $\mathrm{Re}(\mathbf{v}\mathbf{v}^*) \otimes \mathrm{Re}(\mathbf{w}\mathbf{w}^*)$ is clearly symmetric, and it is even positive semidefinite since
    \[
        \mathrm{Re}(\v\v^*) \otimes \mathrm{Re}(\w\w^*) = \big(\mathrm{Re}(\v)\mathrm{Re}(\v)^T + \mathrm{Im}(\v)\mathrm{Im}(\v)^T\big) \otimes \big(\mathrm{Re}(\w)\mathrm{Re}(\w)^T + \mathrm{Im}(\w)\mathrm{Im}(\w)^T\big).
    \]
    It is furthermore the case that $\tr(\mathrm{Re}(\v\v^*) \otimes \mathrm{Re}(\w\w^*)) = 1$, so if we define
    \begin{align*}
        \mathbf{x_1} & := \frac{\mathrm{Re}(\v) \otimes \mathrm{Re}(\w)}{\|\mathrm{Re}(\v) \otimes \mathrm{Re}(\w)\|}, & \mathbf{x_2} & := \frac{\mathrm{Re}(\v) \otimes \mathrm{Im}(\w)}{\|\mathrm{Re}(\v) \otimes \mathrm{Im}(\w)\|}, \\
        \mathbf{x_3} & := \frac{\mathrm{Im}(\v) \otimes \mathrm{Re}(\w)}{\|\mathrm{Im}(\v) \otimes \mathrm{Re}(\w)\|}, & \mathbf{x_4} & := \frac{\mathrm{Im}(\v) \otimes \mathrm{Im}(\w)}{\|\mathrm{Im}(\v) \otimes \mathrm{Im}(\w)\|},
    \end{align*}
    then the quantity in Equation~\eqref{eq:mu_max_conv} is equal to some convex combination of the quantities
    \[
        \tr\big(X(\mathbf{x_1}\mathbf{x_1}^T)\big), \quad \tr\big(X(\mathbf{x_2}\mathbf{x_2}^T)\big), \quad \tr\big(X(\mathbf{x_3}\mathbf{x_3}^T)\big), \quad \text{and} \quad \tr\big(X(\mathbf{x_4}\mathbf{x_4}^T)\big).
    \]
    It follows that there exists some $1 \leq j \leq 4$ such that
    \[
        \mu_{\textup{max}}^{\C}(X) \leq \tr\big(X(\mathbf{x_j}\mathbf{x_j}^T)\big) = \mathbf{x_j}^TX\mathbf{x_j} \leq \mu_{\textup{max}}(X),
    \]
    which completes the proof.
\end{proof}

It might be tempting to conjecture that if $B$ is real symmetric then $\mu_{\textup{min}}(B) = \mu_{\textup{min}}^{\C}(B)$, thus filling in an apparent ``gap'' in Lemma~\ref{lem:bisymmetric_part_hermitian}. We now present an example to show that this is not true:

\begin{example}\label{exx:real_opt_required}
    Consider the symmetric matrix
    \[
        B = \begin{bmatrix}
            0 & 0 & 0 & 1 \\
            0 & 0 & -1 & 0 \\
            0 & -1 & 0 & 0 \\
            1 & 0 & 0 & 0
        \end{bmatrix} \in M_2 \otimes M_2.
    \]
    Since $B^\Gamma = -B$, we know from Lemma~\ref{lem:bisymmetric_part_real} that $(\v \otimes \w)^TB(\v \otimes \w) = 0$ for all $\v, \w \in \R^2$, which implies $\mu_{\textup{min}}(B) = \mu_{\textup{max}}(B) = 0$. On the other hand, $\mu_{\textup{min}}^{\C}(B) = -1$ and $\mu_{\textup{max}}^{\C}(B) = 1$, which can be verified by noting that the eigenvalues of $B$ are $\lambda_{\textup{max}}(B) = 1$ and $\lambda_{\textup{min}}(B) = -1$, and some corresponding eigenvectors are $(1,i) \otimes (1,-i)$ and $(1,i) \otimes (1,i)$, respectively.
\end{example}

\begin{remark}
    Lemma~\ref{lem:bisymmetric_part_hermitian} tells us that the product vector optimizations introduced in Equation~\eqref{eq:main_minmax} and~\eqref{eq:main_minmax_complex} do not depend on the choice of field ($\R$ or $\C$) when real $B \in M_m \otimes M_n$ satisfies $B = B^T = B^\Gamma$. These constraints on $B$ are equivalent to the requirement that $B \in M_m^{\textup{S}} \otimes M_n^{\textup{S}}$.
    
    This vector space is strictly smaller than the vector space of symmetric matrices in $M_m \otimes M_n$: the former has dimension $mn(m+1)(n+1)/4$, while the latter has dimension $mn(mn+1)/2$. This contrasts quite starkly with the complex case, where (here we use $M_n^{\textup{H}}$ to denote the set of $n \times n$ complex Hermitian matrices) $M_m^{\textup{H}} \otimes M_n^{\textup{H}}$ is isomorphic in a natural way to the vector space of Hermitian matrices in $M_m(\C) \otimes M_n(\C)$; they are both real vector spaces with dimension $m^2n^2$. See \cite{Hil08} for further discussion of these sorts of issues.
\end{remark}

Examples~\ref{exx:real_opt_required_even_hermit} and~\ref{exx:real_opt_required} show that we have to be extremely careful when trying to apply our results to optimizations over complex product vectors---the matrix $B$ has to be real, symmetric, and equal to its own partial transpose:

\begin{corollary}\label{cor:num_range_complex}
    Suppose $X \in M_m^\textup{S} \otimes M_n^\textup{S}$ and let $B \in M_m \otimes M_n$ be any matrix for which $X = (B + B^T + B^\Gamma + (B^T)^\Gamma)/4$. Then
    \[
        \mu_{\textup{min}}^{\C}(X) \geq W^{1+i}_{\textup{min}}(B) \quad \text{and} \quad \mu_{\textup{max}}^{\C}(X) \leq W^{1+i}_{\textup{max}}(B).
    \]
\end{corollary}

\begin{proof}
   This follows immediately from combining Theorem~\ref{thm:mu_minmax_from_num_range} and Lemma~\ref{lem:bisymmetric_part_hermitian}.
\end{proof}

\subsection{Application to Entanglement Witnesses}\label{sec:ent_wit}

A standard application of Corollary~\ref{cor:num_range_complex} would be to show that a given matrix $X \in M_m^\textup{H} \otimes M_n^\textup{H}$ is, in the terminology of quantum information theory, an \emph{entanglement witness} \cite{Ter00}, which just means that $X$ has $\mu_{\textup{min}}^{\C}(X) \geq 0$ but is not positive semidefinite. We illustrate this procedure with an example.

\begin{example}\label{exam:choi_like_complex}
    Recall the Choi map $\Phi$ from Equation~\eqref{eq:choi_map_3x3}. It is known that (if $\cdot$ denotes $0$) the matrix
    \begin{align*}
        X = \frac{1}{4}\big(C_\Phi + C_\Phi^T + C_\Phi^\Gamma + (C_\Phi^T)^\Gamma\big) = \frac{1}{2}\left[\begin{array}{@{}ccc|ccc|ccc@{}}
             2 &  \cdot &  \cdot &  \cdot & -1 &  \cdot &  \cdot &  \cdot & -1 \\
             \cdot &  \cdot &  \cdot & -1 &  \cdot &  \cdot &  \cdot &  \cdot &  \cdot \\
             \cdot &  \cdot &  2 &  \cdot &  \cdot &  \cdot & -1 &  \cdot &  \cdot \\\hline
             \cdot & -1 &  \cdot &  2 &  \cdot &  \cdot &  \cdot &  \cdot &  \cdot \\
            -1 &  \cdot &  \cdot &  \cdot &  2 &  \cdot &  \cdot &  \cdot & -1 \\
             \cdot &  \cdot &  \cdot &  \cdot &  \cdot &  \cdot &  \cdot & -1 &  \cdot \\\hline
             \cdot &  \cdot & -1 &  \cdot &  \cdot &  \cdot &  \cdot &  \cdot &  \cdot \\
             \cdot &  \cdot &  \cdot &  \cdot &  \cdot & -1 &  \cdot &  2 &  \cdot \\
            -1 &  \cdot &  \cdot &  \cdot & -1 &  \cdot &  \cdot &  \cdot &  2
        \end{array}\right]
    \end{align*}
    is an entanglement witness, but verifying this fact computationally is difficult. To get an idea of the effectiveness of a computational technique for showing that a matrix is an entanglement witness, we can ask for the smallest $0 \leq c \in \R$ such that the technique succeeds on the input matrix $X + cI$.
    
    Perhaps the computationally simplest such $c$ to find is $c = -\lambda_{\textup{min}}(X)$, which is equal to $(\sqrt{2}-1)/2 \approx 0.2071$ in this case. However, $X + cI$ is technically not an entanglement witness for this value of $c$, since it is actually positive semidefinite. A better (i.e., smaller) $c$ can be computed by the following (much more expensive) semidefinite programming relaxation, which is common in quantum information theory:
    \begin{align}\begin{split}\label{eq:sdp_complex_decomp}
        \textup{minimize:} & \ c \\
        \textup{subject to:} & \ X + cI = Y + Z^\Gamma \\
        & \ Y,Z \succeq O.
    \end{split}\end{align}
    
    This semidefinite program has an optimal value of $c = 2/\sqrt{3} - 1 \approx 0.1547$, which is slightly better than the trivial eigenvalue bound. In quantum information theory terminology, this means that $X + cI$ is a \emph{decomposable} entanglement witness \cite{ATL11} for this value of $c$. The bound of Corollary~\ref{cor:num_range_complex} does just as well and shows that $X + cI$ is an entanglement witness when $c = -W^{1+i}_{\textup{min}}(C_\Phi) = 2/\sqrt{3} - 1 \approx 0.1547$ (see Figure~\ref{fig:choi_itself} and also compare with Figure~\ref{fig:choi_like}) via a fraction of the computational resources.
    
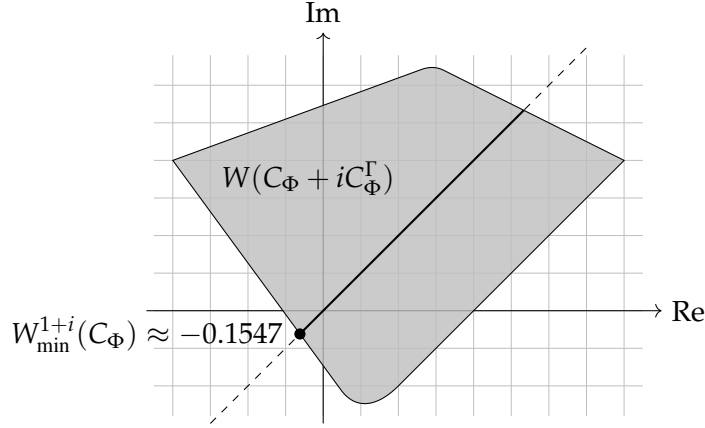
\begin{figure}[htb]
\centering
	\begin{tikzpicture}[scale=2]
		\foreach \x in {-1,-0.75,...,2} {
			\draw[ultra thin, gray!50!white] (\x, -0.7) -- (\x, 1.7);
		}

		\foreach \y in {-0.5,-0.25,...,1.5} {
			\draw[ultra thin, gray!50!white] (-1.125, \y) -- (2.125, \y);
		}

		\draw [->] (-1.175, 0) -- (2.25, 0) node[anchor=west] {$\mathrm{Re}$}; 
		\draw [->] (0, -0.75) -- (0, 1.85) node[anchor=south] {$\mathrm{Im}$}; 

		\draw [dashed] (-0.75,-0.75) -- (1.75,1.75); 

		\draw [fill=gray!50, fill opacity=0.8]
(2.000000, 1.000000) -- (0.494818, -0.505127) -- (0.475131, -0.523630) -- (0.456978, -0.539313) -- (0.440192, -0.552627) -- (0.424625, -0.563946) -- (0.410142, -0.573576) -- (0.396620, -0.581771) -- (0.383954, -0.588739) -- (0.372048, -0.594654) -- (0.360817, -0.599657) -- (0.350188, -0.603869) -- (0.340093, -0.607387) -- (0.330474, -0.610293) -- (0.321279, -0.612656) -- (0.312460, -0.614532) -- (0.303975, -0.615968) -- (0.295786, -0.617004) -- (0.287860, -0.617671) -- (0.280163, -0.617994) -- (0.272669, -0.617995) -- (0.265350, -0.617688) -- (0.258183, -0.617087) -- (0.251144, -0.616198) -- (0.244213, -0.615026) -- (0.237368, -0.613571) -- (0.230591, -0.611831) -- (0.223865, -0.609800) -- (0.217170, -0.607468) -- (0.210490, -0.604823) -- (0.203807, -0.601847) -- (0.197107, -0.598520) -- (0.190371, -0.594816) -- (0.183583, -0.590704) -- (0.176727, -0.586148) -- (0.169787, -0.581104) -- (0.162746, -0.575521) -- (0.155586, -0.569339) -- (0.148293, -0.562487) -- (0.140848, -0.554881) -- (0.133236, -0.546423) -- (0.125442, -0.536997) -- (-1.000000, 1.000000) -- (0.659906, 1.607387) -- (0.669525, 1.610293) -- (0.678720, 1.612656) -- (0.687539, 1.614532) -- (0.696024, 1.615968) -- (0.704213, 1.617004) -- (0.712139, 1.617671) -- (0.719836, 1.617994) -- (0.727330, 1.617995) -- (0.734649, 1.617688) -- (0.741816, 1.617087) -- (0.748855, 1.616198) -- (0.755786, 1.615026) -- (0.762631, 1.613571) -- (0.769408, 1.611831) -- (0.776134, 1.609800) -- (0.782829, 1.607468) -- (0.789509, 1.604823) -- (0.796192, 1.601847) -- cycle;

		\draw [thick] (-0.154700,-0.154700) -- (1.333333,1.333333);
		\node[circle, fill, inner sep=1.5pt] at (-0.154700,-0.154700) {};

		\node at (-0.1, 0.88) {$W(C_\Phi + iC_{\Phi}^\Gamma)$};
		\node[anchor=east, shift={(-0.1,0)}] at (-0.154700,-0.154700) {$W^{1+i}_{\textup{min}}(C_\Phi) \approx -0.1547$};
	\end{tikzpicture}
\caption{The numerical range of $C_\Phi + iC_{\Phi}^\Gamma$, where $\Phi$ is the Choi map on $M_3$.}\label{fig:choi_itself}
\end{figure}
\end{example}

\section{Multipartite Generalization}\label{sec:multipartite}

We now (and for the rest of the paper) return to the case of only considering real matrices and vectors, and we generalize our results to tensor optimizations in which there are more than two tensor factors. Throughout this section, we fix positive integers $p$, $n_1$, $n_2$, $\ldots$, $n_{p}$, and a matrix $B \in M_{n_1} \otimes \cdots \otimes M_{n_p}$. Our goal is to show how our results (Theorems~\ref{thm:mu_minmax_from_num_range} and~\ref{thm:W_diag_equals_p} in particular) can be generalized to bound the quantities
\begin{align}\begin{split}\label{eq:main_minmax_multipartite_general}
    \mu_\textup{min}(B) & \defeq \min_{\mathbf{v_j}\in\R^{n_j}} \big\{ (\mathbf{v_1} \otimes \cdots \otimes \mathbf{v_p})^T B(\mathbf{v_{1}} \otimes \cdots \otimes \mathbf{v_{p}}) : \|\mathbf{v_j}\|=1 \ \text{for all} \ 1 \leq j \leq p \big\} \quad \text{and} \\
    \mu_\textup{max}(B) & \defeq \max_{\mathbf{v_j}\in\R^{n_j}} \big\{ (\mathbf{v_1} \otimes \cdots \otimes \mathbf{v_p})^T B(\mathbf{v_{1}} \otimes \cdots \otimes \mathbf{v_{p}}) : \|\mathbf{v_j}\|=1 \ \text{for all} \ 1 \leq j \leq p \big\}.
\end{split}\end{align}
If $p = 2$ then these quantities simplify to exactly the quantities~\eqref{eq:main_minmax} that we considered earlier.

Now that we are considering tensor optimizations on more than just two systems, we need to introduce two modifications of tools that we used previously before we can generalize Theorems~\ref{thm:mu_minmax_from_num_range} and~\ref{thm:W_diag_equals_p}. First, we cannot simply make use of the partial transpose of Equation~\eqref{eq:part_trans}. Instead, we can now transpose any combination of the $p$ subsystems, not just the second subsystem. That is, for $1 \leq j \leq p$ we define the $j$-th partial transpose of a matrix $B = \sum_\ell X_{1,\ell} \otimes X_{2,\ell} \otimes \cdots \otimes X_{p,\ell}$ as
\begin{align}\label{eq:alt_part_trans}
    \Gamma_j(B) \defeq \sum_\ell X_{1,\ell} \otimes \cdots \otimes X_{j-1,\ell} \otimes X_{j,\ell}^T \otimes X_{j+1,\ell} \otimes \cdots \otimes X_{p,\ell}.
\end{align}
Slightly more generally, if $S \subseteq [p]$ is any subset of $[p] = \{1,2,\ldots,p\}$, then we define $\Gamma_S$ to be the composition of $\Gamma_j$ for each $j \in S$.

Second, instead of the numerical range, we now make use of the \emph{joint} numerical range, which is defined for $A_1,A_2,\ldots,A_k \in M_n^\textup{S}$ by
\begin{align}\label{eq:joint_def}
    W(A_1,A_2,\ldots,A_k) \defeq \big\{(\v^*A_1\v, \v^*A_2\v, \ldots, \v^*A_k\v) : \v \in \C^n, \|\v\| = 1 \big\}.
\end{align}
In the special case when $k = 2$, the joint numerical range $W(A_1,A_2)$ is exactly equal to the (regular, non-joint) numerical range $W(A_1 + iA_2)$, if we identify $\R^2$ with $\C$ in the usual way. Unlike the numerical range, however, when $k \geq 4$ or $(k,n) = (3,2)$ the joint numerical range is not convex in general (see \cite{BL91}, for example).

Our bounds on $\mu_\textup{max}(B)$ and $\mu_\textup{min}(B)$ will depend on the following joint numerical range of several partial transpositions of of $B$. In particular, if $P = \{S_1,S_2,\ldots,S_k\} \subseteq \mathcal{P}([p])$ is any non-empty member of the power set of $[p]$, then we define
\begin{align}\label{eq:W_diag_multi}
    W^{P,\bm{1}}(B) \defeq \big\{c\in\R : (c,c,\ldots,c) \in W\big(\Gamma_{S_1}(B),\Gamma_{S_2}(B),\ldots,\Gamma_{S_k}(B)\big) \big\}.
\end{align}

As with the bipartite case of $W^{1+i}(B)$ from Equation~\eqref{eq:def_W_diag}, $W^{P,\bm{1}}(B)$ is always non-empty since it contains, for example $(\e_1 \otimes \cdots \otimes \e_1)^TB(\e_1 \otimes \cdots \otimes \e_1)$. Also similar to the bipartite case, we denote its extreme values by
\[
    W^{P,\bm{1}}_\textup{min}(B) \defeq \min\big\{c \in W^{P,\bm{1}}(B)\big\} \quad \text{and}
    \quad W^{P,\bm{1}}_\textup{max}(B) \defeq \max\big\{c \in W^{P,\bm{1}}(B)\big\}.
\]
Since $W\big(\Gamma_{S_1}(B),\Gamma_{S_2}(B),\ldots,\Gamma_{S_k}(B)\big)$ is compact~\cite[Property~(F)]{LR02}, so too is $W^{P,\bm{1}}(B)$, and thus these bounds are attained and well defined. We are finally in a position to state the main result of this section:

\begin{theorem}\label{thm:multipartite_mu_general}
    Let $B \in M_{n_1} \otimes \cdots \otimes M_{n_p}$ and let $P \subseteq \mathcal{P}([p])$ be non-empty. Then
    \begin{align*}
        \mu_\textup{min}(B) \geq W^{P,\bm{1}}_\textup{min}(B) \quad \text{and} \quad \mu_\textup{max}(B) \leq W^{P,\bm{1}}_\textup{max}(B).
    \end{align*}
\end{theorem}

For example, if $p = 2$ and $P = \big\{\{\},\{2\}\big\}$, then this theorem recovers exactly Theorem~\ref{thm:mu_minmax_from_num_range}, since the two members of $P$ correspond to taking the partial transpose on no subsystems (i.e., $B$ itself) and taking the partial transpose on the 2nd subsystem (which we simply called $B^\Gamma$).

In general, choosing a larger set $P$ results in a tighter bound on $\mu_\textup{min}(B)$ and $\mu_\textup{min}(B)$ in Theorem~\ref{thm:multipartite_mu_general}, at the expense of increased computational resources. However, there is never any advantage to choosing $P$ so that $|P| > 2^{p-1}$ (which is why we chose $P$ with $|P| = 2$ in the $p = 2$ case, for example, rather than choosing $P = \big\{\{\},\{1\},\{2\},\{1,2\}\big\}$ with $|P| = 4$). To see why this is the case, notice that if $S \in P$ then there is nothing to be gained by having $[p] \setminus S \in P$, since $\Gamma_{[p] \setminus S}(B) = (T \circ \Gamma_S)(B)$, so $\x^T\Gamma_{[p] \setminus S}(B)\x = \x^T\Gamma_S(B)\x$ for all $\x \in R^{n_1\cdots n_p}$, so $W^{P,\bm{1}}(B)$ does not depend on whether or not $[p] \setminus S \in P$.

\begin{proof}[Proof of Theorem~\ref{thm:multipartite_mu_general}]
   We only prove the rightmost inequality for brevity, as the proof of the left inequality is almost identical.
   
   By compactness of the set of unit product vectors, we know that there exists $\x = \mathbf{v_1} \otimes \cdots \otimes \mathbf{v_p} \in \R^{n_1} \otimes \cdots \otimes \R^{n_p}$ such that $\norm{\x} = 1$ and $\x^TB\x = \mu_\textup{max}(B)$. Then for all $S \in P$ we have $\Gamma_S(\x\x^T) = \x\x^T$, so
    \[
        \x^T\Gamma_S(B)\x = \tr\big(\Gamma_S(B)\x\x^T\big) = \tr\big(B\Gamma_S(\x\x^T)\big) = \tr\big(B\x\x^T\big) = \x^TB\x = \mu_\textup{max}(B),
    \]
    where we used the fact that each partial transpose map $\Gamma_S$ is its own dual in the trace inner product (i.e., $\tr(\Gamma_S(X)^TY) = \tr(X^T\Gamma_S(Y))$ for all matrices $X$ and $Y$). It follows that $\mu_\textup{max}(B) \in W^{P,\bm{1}}(B)$, so $\mu_\textup{max}(B) \leq W^{P,\bm{1}}_\textup{max}(B)$.
\end{proof}

Our final result generalizes Theorem~\ref{thm:W_diag_equals_p} to this multipartite setting (once again, that result comes from choosing $p = 2$ and $P = \big\{\{\},\{2\}\big\}$ in this generalization). Note that we need $B$ to be symmetric in this theorem so that we can ensure that the eigenvalues considered are real.

\begin{theorem}\label{thm:multi_affine_minmax}
    Let $B \in M_{n_1} \otimes \cdots \otimes M_{n_p}$ be symmetric and let $P = \{S_1,S_2,\ldots,S_k\} \subseteq \mathcal{P}([p])$ be non-empty. Then
    \begin{align}\label{eq:min_max_eig}
        W^{P,\bm{1}}_\textup{min}(B)
        &= \max_{p_j \in \R}\bigg\{\lambda_{\textup{min}}\bigg(\sum_{j=1}^k p_j\Gamma_{S_j}(B)\bigg) : \sum_{j=1}^k p_j = 1\bigg\} \quad \text{and} \\\label{eq:max_min_eig}
        W^{P,\bm{1}}_\textup{max}(B)
        &= \min_{p_j \in \R}\bigg\{\lambda_{\textup{max}}\bigg(\sum_{j=1}^k p_j\Gamma_{S_j}(B)\bigg) : \sum_{j=1}^k p_j = 1\bigg\}.
    \end{align}
\end{theorem}

\begin{proof}
    We only prove the lower equality for sake of brevity. Let $\x \in \R^{n_1\cdots n_p}$ be a unit vector for which $\x^*\Gamma_{S_j}(B)\x = W^{P,\bm{1}}_\textup{max}(B)$ for all $1 \leq j \leq k$. If $\sum_{j=1}^k p_j = 1$ then
    \[
        \x^* \bigg( \sum_{j=1}^k p_j\Gamma_{S_j}(B)\bigg) \x
        = \sum_{j=1}^k p_j \big(\x^* \Gamma_{S_j}(B) \x\big)
        = \sum_{j=1}^k p_j W^{P,\bm{1}}_\textup{max}(B) = W^{P,\bm{1}}_\textup{max}(B),
    \]
    and thus $W^{P,\bm{1}}_\textup{max}(B) \leq \lambda_{\textup{max}}\big(\sum_{j=1}^k p_j\Gamma_{S_j}(B)\big)$.
    
    To see that opposite inequality is attained for some particular choice of $p_1, p_2, \ldots, p_k$ (and thus complete the proof), we make use of Farkas' Lemma from semidefinite programming (see \cite[Lemma~3.3]{Lov03}, for example), which we state here for clarity:\medskip
    
    \noindent\textbf{Farkas' Lemma.} Let $A_1,\ldots,A_m$ and $C$ be real symmetric matrices. Then there exist $z_1,\ldots,z_m \in \R$ such that $z_1A_1 + \cdots + z_mA_m - C$ is positive definite if and only if there does \emph{not} exist a non-zero (symmetric) PSD matrix $Y$ such that $\tr(CY) \geq 0$ and $\tr(A_jY) = 0$ for all $1 \leq j \leq m$.\medskip
    
    To make use of this lemma, we fix $\varepsilon > 0$, set $m = k-1$, and choose the following matrices $A_1,\ldots,A_m$, and $C$:
    \begin{itemize}
        \item $A_j = \Gamma_{S_1}(B) - \Gamma_{S_{j+1}}(B)$ for all $1 \leq j \leq m$, and
        
        \item $C = \Gamma_{S_1}(B) - \big(W^{P,\bm{1}}_\textup{max}(B) + \varepsilon\big)I$.
    \end{itemize}
    We claim that there does not exist a non-zero (symmetric) PSD matrix $Y$ such that $\tr(CY) \geq 0$ and $\tr(A_jY) = 0$ for all $1 \leq j \leq m$. To verify this claim, it suffices (thanks to the spectral decomposition) to consider the case when $Y = \y\y^T$ has rank $1$, and by rescaling $Y$ we can assume without loss of generality that $\|\y\| = 1$. Then $\tr(A_jY) = 0$ for all $1 \leq j \leq m$ implies $\y^T\Gamma_{S_1}(B)\y = \y^T\Gamma_{S_{j+1}}(B)\y$ for all $1 \leq j \leq m$, so $\y^T\Gamma_{S_1}(B)\y \leq W^{P,\bm{1}}_\textup{max}(B)$. The condition $\tr(CY) \geq 0$ implies $\y^T\Gamma_{S_1}(B)\y \geq W^{P,\bm{1}}_\textup{max}(B) + \varepsilon$, which cannot also be true, so $Y$ does not exist.
    
    Farkas' Lemma thus tells us that there exist $z_1,\ldots,z_m \in \R$ such that $z_1A_1 + \cdots + z_mA_m - C$ is positive definite. Plugging in our particular choices of $A_1, \ldots, A_m$ and $C$ shows that there exist $z_1,\ldots,z_m \in \R$ for which
    \[
        \big(W^{P,\bm{1}}_\textup{max}(B) + \varepsilon\big)I + \left(\sum_{j=1}^m z_j - 1\right)\Gamma_{S_1}(B) - \sum_{j=1}^m z_j\Gamma_{S_{j+1}}(B)
    \]
    is positive definite. For $1 \leq j \leq m = k-1$ we set $p_j = z_j$, and we set $p_k = 1 - \sum_{j=1}^{m}z_j$ so that $\sum_{j=1}^{k}p_j = 1$. It follows that the matrix
    \[
        \big(W^{P,\bm{1}}_\textup{max}(B) + \varepsilon\big)I - \sum_{j=1}^k p_j\Gamma_{S_{j}}(B)
    \]
    is positive definite, so
    \[
        W^{P,\bm{1}}_\textup{max}(B) + \varepsilon > \lambda_{\textup{max}}\bigg(\sum_{j=1}^k p_j\Gamma_{S_j}(B))\bigg).
    \]
    Since $\varepsilon > 0$ was arbitrary and the function that we are minimizing in Equation~\eqref{eq:max_min_eig} is a convex function of $p_1,\ldots,p_k$, we conclude that there exist $p_1,\ldots,p_k$ for which $\sum_{j=1}^k p_j = 1$ and $W^{P,\bm{1}}_\textup{max}(B) \geq \lambda_{\textup{max}}(\sum_{j=1}^k p_j\Gamma_{S_j}(B)))$, completing the proof.
\end{proof}

\noindent \textbf{Acknowledgements.} We thank Peter Selinger for asking a helpful question that led to Remark~\ref{rem:why1i}. N.J.\ was supported by NSERC Discovery Grant RGPIN-2022-04098.

\bibliographystyle{alpha}
\bibliography{bib}

\appendix
\section{Implementation of Theorem~\ref{thm:mu_minmax_from_num_range}}\label{sec:implementation}

We now describe two methods of numerically computing $W^{1+i}_\textup{min}(B)$ and $W^{1+i}_\textup{max}(B)$, both of which we have implemented in MATLAB \cite{Pip22}.

The conceptually easiest way to approximate the numerical range of a given matrix (and thus compute $W^{1+i}_\textup{min}(B)$ and $W^{1+i}_\textup{max}(B)$) is to simply generate a large number of random unit vectors $\v$ and compute $\v^*A\v$ for each of them. However, this method only yields a vague semblance of some points on the interior of $W(A)$, and typically does not properly demonstrate the actual shape of the numerical range unless an astronomical number of random unit vectors are chosen. A much better method is demonstrated in~\cite[Chapter~1.5]{HJ91}, which we outline here.

Suppose $A \in M_n(\C)$ has Cartesian decomposition $A = H(A) + S(A)$, where
\[
    H(A)=\frac{1}{2}(A+A^*) \quad \text{and} \quad S(A) = \frac{1}{2}(A - A^*)
\]
denote the Hermitian and skew-Hermitian pieces of $A$, respectively. Then $H(A)$ acts as a sort of orthogonal projection on to the real line for the numerical range: $W(H(A)) = \mathrm{Re}(W(A))$. It is also the case that $W(H(A))$ is simply the closed interval from $\lambda_\textup{min}(H(A))$ to $\lambda_\textup{max}(H(A))$, since Hermitian matrices are normal, and the numerical range of a normal matrix is the convex hull of its eigenvalues. When we combine these facts with the observation that $e^{-i\theta}W(e^{i\theta}A) = W(A)$, we can obtain a boundary point of $W(A)$ whose tangent is normal to any given angle $\theta$
by first rotating $A$, computing the largest eigenvalue of the Hermitian part of the new matrix, and conjugating the original $A$ by the corresponding eigenvector.
Some of these steps are illustrated in Figure~\ref{fig:compute_num_range}, where $\x$ is the eigenvector corresponding to the eigenvalue $\lambda_\textup{max}(H(e^{i\theta}A))$.

\begin{figure}[tbh]
\begin{centering}
	\begin{tikzpicture}[scale=0.6]
		\draw [<->] (-10, 0) -- (10, 0) node[anchor=west] {$\mathrm{Re}$}; 
		\draw [->] (0, 0) -- (0, 6) node[anchor=south] {$\mathrm{Im}$}; 
		
		\draw [fill=gray!50, fill opacity=0.8]
(2.000000, 0.000000) -- (1.998245, -0.083751) -- (1.992985, -0.167355) -- (1.984229, -0.250666) -- (1.971992, -0.333537) -- (1.956295, -0.415823) -- (1.937166, -0.497379) -- (1.914638, -0.578063) -- (1.888752, -0.657733) -- (1.859552, -0.736249) -- (1.827090, -0.813473) -- (1.791423, -0.889270) -- (1.752613, -0.963507) -- (1.710728, -1.036054) -- (1.665842, -1.106783) -- (1.618033, -1.175570) -- (1.567386, -1.242295) -- (1.513990, -1.306841) -- (1.457937, -1.369094) -- (1.399326, -1.428945) -- (1.338261, -1.486289) -- (1.274847, -1.541026) -- (1.209198, -1.593059) -- (1.141427, -1.642298) -- (1.071653, -1.688655) -- (1.000000, -1.732050) -- (0.926592, -1.772407) -- (0.851558, -1.809654) -- (0.775031, -1.843726) -- (0.697144, -1.874563) -- (0.618033, -1.902113) -- (0.537839, -1.926325) -- (0.456701, -1.947157) -- (0.374762, -1.964574) -- (0.292166, -1.978544) -- (0.209056, -1.989043) -- (0.125581, -1.996053) -- (0.041884, -1.999561) -- (-0.041884, -1.999561) -- (-0.125581, -1.996053) -- (-0.209056, -1.989043) -- (-0.292166, -1.978544) -- (-0.374762, -1.964574) -- (-5.228350, -0.973578) -- (-5.268919, -0.963162) -- (-5.309016, -0.951056) -- (-5.348572, -0.937281) -- (-5.387515, -0.921863) -- (-5.425779, -0.904827) -- (-5.463296, -0.886203) -- (-5.500000, -0.866025) -- (-5.535826, -0.844327) -- (-5.570713, -0.821149) -- (-5.604599, -0.796529) -- (-5.637423, -0.770513) -- (-5.669130, -0.743144) -- (-5.699663, -0.714472) -- (-5.728968, -0.684547) -- (-5.756995, -0.653420) -- (-5.783693, -0.621147) -- (-5.809016, -0.587785) -- (-5.832921, -0.553391) -- (-5.855364, -0.518027) -- (-5.876306, -0.481753) -- (-5.895711, -0.444635) -- (-5.913545, -0.406736) -- (-5.929776, -0.368124) -- (-5.944376, -0.328866) -- (-5.957319, -0.289031) -- (-5.968583, -0.248689) -- (-5.978147, -0.207911) -- (-5.985996, -0.166768) -- (-5.992114, -0.125333) -- (-5.996492, -0.083677) -- (-5.999122, -0.041875) -- (-6.000000, 0.000000) -- (-5.999122, 0.041875) -- (-5.996492, 0.083677) -- (-5.992114, 0.125333) -- (-5.985996, 0.166768) -- (-5.978147, 0.207911) -- (-5.968583, 0.248689) -- (-5.957319, 0.289031) -- (-5.944376, 0.328866) -- (-5.929776, 0.368124) -- (-5.913545, 0.406736) -- (-5.895711, 0.444635) -- (-5.876306, 0.481753) -- (-5.855364, 0.518027) -- (-5.832921, 0.553391) -- (-5.809016, 0.587785) -- (-5.783693, 0.621147) -- (-5.756995, 0.653420) -- (-5.728968, 0.684547) -- (-5.699663, 0.714472) -- (-5.669130, 0.743144) -- (-5.637423, 0.770513) -- (-5.604599, 0.796529) -- (-5.570713, 0.821149) -- (-5.535826, 0.844327) -- (-5.500000, 0.866025) -- (-5.463296, 0.886203) -- (-5.425779, 0.904827) -- (-5.387515, 0.921863) -- (-5.348572, 0.937281) -- (-5.309016, 0.951056) -- (-5.268919, 0.963162) -- (-5.228350, 0.973578) -- (-0.374762, 1.964574) -- (-0.292166, 1.978544) -- (-0.209056, 1.989043) -- (-0.125581, 1.996053) -- (-0.041884, 1.999561) -- (0.041884, 1.999561) -- (0.125581, 1.996053) -- (0.209056, 1.989043) -- (0.292166, 1.978544) -- (0.374762, 1.964574) -- (0.456701, 1.947157) -- (0.537839, 1.926325) -- (0.618033, 1.902113) -- (0.697144, 1.874563) -- (0.775031, 1.843726) -- (0.851558, 1.809654) -- (0.926592, 1.772407) -- (0.999999, 1.732050) -- (1.071653, 1.688655) -- (1.141427, 1.642298) -- (1.209198, 1.593059) -- (1.274847, 1.541026) -- (1.338261, 1.486289) -- (1.399326, 1.428945) -- (1.457937, 1.369094) -- (1.513990, 1.306841) -- (1.567386, 1.242295) -- (1.618033, 1.175570) -- (1.665842, 1.106783) -- (1.710728, 1.036054) -- (1.752613, 0.963507) -- (1.791423, 0.889270) -- (1.827090, 0.813473) -- (1.859552, 0.736249) -- (1.888752, 0.657733) -- (1.914638, 0.578063) -- (1.937166, 0.497379) -- (1.956295, 0.415823) -- (1.971992, 0.333537) -- (1.984229, 0.250666) -- (1.992985, 0.167355) -- (1.998245, 0.083751) -- cycle;

		\draw [fill=gray!50, fill opacity=0.8]
(4.535533, 3.535533) -- (4.534656, 3.493658) -- (4.532026, 3.451856) -- (4.527648, 3.410200) -- (4.521529, 3.368765) -- (4.513681, 3.327622) -- (4.504117, 3.286844) -- (4.492853, 3.246502) -- (4.479910, 3.206667) -- (4.465310, 3.167409) -- (4.449079, 3.128797) -- (4.431245, 3.090898) -- (4.411840, 3.053780) -- (4.390898, 3.017506) -- (1.665842, -1.106783) -- (1.618033, -1.175570) -- (1.567386, -1.242295) -- (1.513990, -1.306841) -- (1.457937, -1.369094) -- (1.399326, -1.428945) -- (1.338261, -1.486289) -- (1.274847, -1.541026) -- (1.209198, -1.593059) -- (1.141427, -1.642298) -- (1.071653, -1.688655) -- (1.000000, -1.732050) -- (0.926592, -1.772407) -- (0.851558, -1.809654) -- (0.775031, -1.843726) -- (0.697144, -1.874563) -- (0.618033, -1.902113) -- (0.537839, -1.926325) -- (0.456701, -1.947157) -- (0.374762, -1.964574) -- (0.292166, -1.978544) -- (0.209056, -1.989043) -- (0.125581, -1.996053) -- (0.041884, -1.999561) -- (-0.041884, -1.999561) -- (-0.125581, -1.996053) -- (-0.209056, -1.989043) -- (-0.292166, -1.978544) -- (-0.374762, -1.964574) -- (-0.456701, -1.947157) -- (-0.537839, -1.926325) -- (-0.618033, -1.902113) -- (-0.697144, -1.874563) -- (-0.775031, -1.843726) -- (-0.851558, -1.809654) -- (-0.926592, -1.772407) -- (-1.000000, -1.732050) -- (-1.071653, -1.688655) -- (-1.141427, -1.642298) -- (-1.209198, -1.593059) -- (-1.274847, -1.541026) -- (-1.338261, -1.486289) -- (-1.399326, -1.428945) -- (-1.457937, -1.369094) -- (-1.513990, -1.306841) -- (-1.567386, -1.242295) -- (-1.618033, -1.175570) -- (-1.665842, -1.106783) -- (-1.710728, -1.036054) -- (-1.752613, -0.963507) -- (-1.791423, -0.889270) -- (-1.827090, -0.813473) -- (-1.859552, -0.736249) -- (-1.888752, -0.657733) -- (-1.914638, -0.578063) -- (-1.937166, -0.497379) -- (-1.956295, -0.415823) -- (-1.971992, -0.333537) -- (-1.984229, -0.250666) -- (-1.992985, -0.167355) -- (-1.998245, -0.083751) -- (-2.000000, 0.000000) -- (-1.998245, 0.083751) -- (-1.992985, 0.167355) -- (-1.984229, 0.250666) -- (-1.971992, 0.333537) -- (-1.956295, 0.415823) -- (-1.937166, 0.497379) -- (-1.914638, 0.578063) -- (-1.888752, 0.657733) -- (-1.859552, 0.736249) -- (-1.827090, 0.813473) -- (-1.791423, 0.889270) -- (-1.752613, 0.963507) -- (-1.710728, 1.036054) -- (-1.665842, 1.106783) -- (-1.618033, 1.175570) -- (-1.567386, 1.242295) -- (-1.513990, 1.306841) -- (-1.457937, 1.369094) -- (-1.399326, 1.428945) -- (-1.338261, 1.486289) -- (-1.274847, 1.541026) -- (-1.209198, 1.593059) -- (-1.141427, 1.642298) -- (2.999707, 4.379861) -- (3.035533, 4.401559) -- (3.072237, 4.421737) -- (3.109754, 4.440360) -- (3.148018, 4.457397) -- (3.186961, 4.472815) -- (3.226516, 4.486590) -- (3.266614, 4.498696) -- (3.307183, 4.509112) -- (3.348152, 4.517821) -- (3.389450, 4.524806) -- (3.431005, 4.530055) -- (3.472743, 4.533560) -- (3.514591, 4.535314) -- (3.556476, 4.535314) -- (3.598324, 4.533560) -- (3.640062, 4.530055) -- (3.681616, 4.524806) -- (3.722915, 4.517821) -- (3.763884, 4.509112) -- (3.804453, 4.498696) -- (3.844550, 4.486590) -- (3.884105, 4.472815) -- (3.923049, 4.457397) -- (3.961313, 4.440360) -- (3.998829, 4.421737) -- (4.035533, 4.401559) -- (4.071360, 4.379861) -- (4.106247, 4.356683) -- (4.140133, 4.332063) -- (4.172957, 4.306047) -- (4.204664, 4.278678) -- (4.235197, 4.250006) -- (4.264502, 4.220081) -- (4.292528, 4.188954) -- (4.319227, 4.156681) -- (4.344550, 4.123319) -- (4.368455, 4.088925) -- (4.390898, 4.053560) -- (4.411840, 4.017287) -- (4.431245, 3.980169) -- (4.449079, 3.942270) -- (4.465310, 3.903658) -- (4.479910, 3.864400) -- (4.492853, 3.824565) -- (4.504117, 3.784223) -- (4.513681, 3.743445) -- (4.521529, 3.702302) -- (4.527648, 3.660867) -- (4.532026, 3.619211) -- (4.534656, 3.577409) -- cycle;

		\draw [thick] (-2,0) -- (4.535533,0);
		\node[circle, fill, inner sep=1.5pt] at (-2, 0) {};
		\node[circle, fill, inner sep=1.5pt] at (4.535533, 0) {};

		\node[circle, fill, inner sep=1.5pt] at (4.535533, 3.535533) {};
		\node[circle, fill, inner sep=1.5pt] at (-5.707106, 0.707106) {};

		\draw [dashed] (4.535533, 5) -- (4.535533, -1.8);
		\draw [dashed] (-6.707106, -0.292894) -- (-4.707106, 1.707106);

		\node at (-3.5, 0.6) {$W(A)$};
		\node at (2.25, 2.5) {$W(e^{i\theta}A)$};
		\draw[decorate,decoration={brace,amplitude=7pt}] (4.535533, -2.2) --node[anchor=north,shift={(0,-0.25)}]{$W(H(e^{i\theta}A))$} (-2, -2.2);
		\node[anchor=south west] at (4.535533, 0) {$\lambda_\textup{max}(H(e^{i\theta}A))$};
		\node[anchor=west] at (4.59, 3.535533) {$\x^*(e^{i\theta}A)\x$};
		\node[anchor=south east] at (-5.707106, 0.707106) {$\x^*A\x$};
	\end{tikzpicture}
\caption{An illustration of the steps used to compute a boundary point of $W(A)$.}
\label{fig:compute_num_range}
\end{centering}
\end{figure}
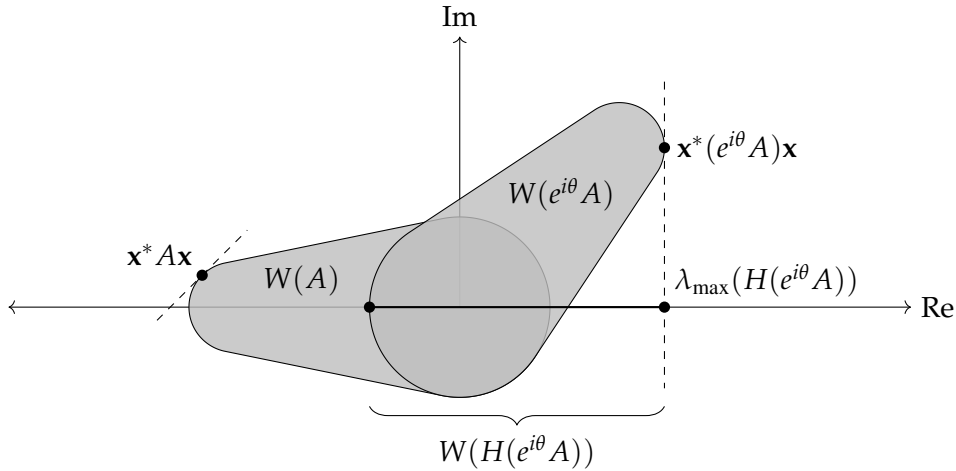

Once we have a set of boundary points, we can compute an inner approximation of the numerical range by simply connecting them,
or since we know the angle of the tangent at each point, we can compute an outer approximation
by finding the intersections of each half-plane described by a tangent.

Note however that evenly spacing choices of $\theta$ does \emph{not} determine evenly space points along the boundary of $W(A)$;
in particular if the boundary of $W(A)$ contains a flat region, this method will not yield any points along it,
it will only result in the endpoints being computed.

Our method for computing $W^{1+i}_\textup{min}(B)$ and $W^{1+i}_\textup{max}(B)$ efficiently is then to perform a binary search on values of $\theta$ to compute boundary points as close as possible both below and above the line of slope $1$.
This process may not yield boundary points arbitrarily close to the line, but it does yield points whose tangents are at arbitrarily close angles. If those points themselves are not sufficiently close, the tangent angle at $W^{1+i}_\textup{min}(B)$ or $W^{1+i}_\textup{max}(B)$ must lie between that of the surrounding points, and since their angles differ only by an arbitrarily small amount, we can approximate the value in question by drawing a straight line between them.

Another completely different method of computing $W^{1+i}_\textup{min}(B)$ and $W^{1+i}_\textup{max}(B)$ is suggested by Theorem~\ref{thm:W_diag_equals_p}. Since $\lambda_{\textup{min}}(pB + (1-p)B^\Gamma)$ and $\lambda_{\textup{max}}(pB + (1-p)B^\Gamma)$ are concave and convex functions of $p$, respectively, we can perform ternary search on $p$ to numerically find the value that maximizes the minimum eigenvalue or minimizes the maximum eigenvalue. We have found that, in practice, this method performs comparably to the method of computing the boundary of $W(A)$.

\end{document}